\documentclass[11pt,a4paper]{article}
\usepackage[T1]{fontenc}
\usepackage{lmodern}
\usepackage{amsmath,amssymb,amsthm,mathtools}
\usepackage[margin=27mm]{geometry}
\usepackage{microtype}
\usepackage{enumitem}
\usepackage{needspace}
\usepackage{authblk}
\usepackage{cite}
\usepackage[hidelinks,hypertexnames=false]{hyperref}
\numberwithin{equation}{section}
\newtheorem{theorem}{Theorem}[section]
\newtheorem{lemma}[theorem]{Lemma}
\newtheorem{proposition}[theorem]{Proposition}

\theoremstyle{definition}

\theoremstyle{remark}
\newtheorem*{remark}{Remark}
\newcommand{\G}{\mathbf{G}}
\newcommand{\Z}{\mathbb{Z}}
\newcommand{\ind}{\mathbf{1}}
\renewcommand{\phi}{\varphi}
\setlist[enumerate]{label=(\roman*),itemsep=3pt}
\title{A Bondy-type theorem for rainbow pancyclicity in graph systems}
\author[1]{Ailian Chen\thanks{Corresponding author: \texttt{elian1425@fzu.edu.cn}.}}
\author[1]{Liping Zhang\thanks{ \texttt{1519036064@qq.com}.}}
\affil[1]{School of Mathematics and Statistics, Fuzhou University,\\
Fuzhou 350108, China}
\date{}
\hypersetup{
  pdftitle={A Bondy-type theorem for rainbow pancyclicity in graph systems},
  pdfauthor={Ailian Chen and Hao Chen},
  pdfsubject={Rainbow pancyclicity in graph systems},
  pdfkeywords={graph system, rainbow cycle, pancyclicity, Hamilton cycle, extremal graph theory}
}
\begin{document}
\maketitle
\begin{abstract}
We establish a Hamiltonian-to-pancyclic analogue of Bondy's theorem for graph
systems under an aggregate degree condition.  Let
$\G=(G_1,\ldots,G_n)$ be a graph system on a common $n$-vertex set $V$, and
write
$\delta(v)=\min_{i\in[n]}d_{G_i}(v)$.  If $\G$ contains a rainbow
Hamilton cycle and
\[
  \sum_{v\in V}\delta(v)\ge
  \left\lceil\frac{n^2}{2}\right\rceil-1,
\]
then $\G$ is rainbow pancyclic, unless $n$ is even and every member is the
same balanced complete bipartite graph.  For even $n$ the threshold is exact
at the integer level.

Unlike the usual transversal Dirac- or Ore-type hypotheses, our condition is
not layerwise: the member attaining $\delta(v)$ may depend on $v$, and some
vertices may have $\delta(v)<n/2$.  Relative to a fixed rainbow Hamilton
cycle, we count shortcuts whose colors are released by the Hamilton arcs they
replace.  A missing cycle length forces complementary shortcut supports to
cross-intersect.  A counting gap settles even shortening, while equality or
near equality in odd shortening yields a distance-two exchange whose orbits
force the balanced bipartite obstruction.  At the lower integer threshold an
exact defect identity shows that only one or two units of slack are available.
\end{abstract}


\section{Introduction}

All graphs are finite, simple, and undirected.  A \emph{graph system} is an
indexed family $\G=(G_1,\ldots,G_n)$ of graphs on a common vertex set $V$,
where $|V|=n$.  We regard the index $i$ as color $i$.  A subgraph is \emph{rainbow} if
its edges can be assigned distinct colors $i\in[n]$ so that an edge assigned
color $i$ belongs to $G_i$.  Thus a rainbow Hamilton cycle is a Hamilton
transversal.  We call
$\G$ \emph{rainbow pancyclic} if it contains a rainbow cycle of every
length from $3$ to $n$.  For $v\in V$, write
$
  \delta(v):=\min_{i\in[n]} d_{G_i}(v).
$
We refer to $\delta(v)$ as the \emph{system degree} of $v$.  Notice that it
is a minimum over the members of the system, not the degree in their
intersection.

Bondy's theorem \cite{Bondy1971} says that a Hamiltonian $n$-vertex graph
with at least $n^2/4$ edges is pancyclic unless it is the balanced complete
bipartite graph.  It is one of the cleanest Hamiltonian-to-pancyclic results:
once Hamiltonicity is already present, a global density condition forces the
entire cycle spectrum apart from a rigid extremal family.  The ordinary-graph
density threshold was subsequently refined for non-bipartite Hamiltonian
graphs by H\"aggkvist, Faudree and Schelp
\cite{HaggkvistFaudreeSchelp1981}; related cycle-structure results include
those of Schmeichel and Hakimi \cite{SchmeichelHakimi1988} and Keevash and
Sudakov \cite{KeevashSudakov2010}.

The corresponding transversal theory begins with Hamiltonicity.  Aharoni,
DeVos, Gonz\'alez Hermosillo de la Maza, Montejano and \v{S}\'amal
\cite{Aharoni2020} proved a rainbow Mantel theorem and posed a rainbow
Dirac conjecture.  Joos and Kim \cite{Joos2020} proved the exact Dirac
statement: if $\delta(G_i)\ge n/2$ for every $i$, then the system contains
a rainbow Hamilton cycle.  Since then, the Hamiltonian theory has been
developed in several directions, including general spanning methods
\cite{ChengHanWangWang2023}, stability near the Dirac threshold
\cite{ChengStaden2025}, an exact description of near-Dirac non-Hamiltonian
systems \cite{ChengSunWangWei2026}, and an Ore-type Hamiltonian transversal
criterion due to Liu, Chen and Ma \cite{LiuChenMa2025}.  We refer to Sun,
Wang and Wei \cite{SunWangWeiSurvey2026} for a recent survey.

Cycle-spectrum questions have developed in parallel.  Bradshaw
\cite{Bradshaw2021} obtained a transversal bipancyclicity theorem for
balanced bipartite graph families.  Cheng, Wang and Zhao \cite{Cheng2021}
proved an asymptotic rainbow Hamilton theorem and obtained rainbow cycles
of all lengths $3,\ldots,n-1$ under the stronger layerwise assumption
$\delta(G_i)\ge(n+1)/2$.  Li, Li and Li \cite{LiLiLi2023} studied
vertex-pancyclicity and panconnectedness, and later proved the exact
Dirac-type rainbow pancyclicity theorem \cite{LiLiLi2024}.  More recently,
Li, Wang and Yan \cite{LiWangYan2026} obtained rainbow pancyclicity and
vertex-pancyclicity under a system Ore-type condition.  Sharp transversal
panconnectedness was obtained by Sun, Wang and Wei \cite{SunWangWei2025},
and Ma, You and Zhang \cite{MaYouZhang2026} subsequently strengthened the
minimum-degree theory of rainbow panconnectivity.  Ma and Cai
\cite{MaCai2026} considered rainbow chorded pancyclicity.  Density conditions
of a different kind have also appeared; for example, Liu, Zhang and Wang
\cite{LiuZhangWang2025} studied Hamiltonian structures under large
edge-number assumptions.

Our point of view is different from the Dirac- and Ore-type results just
cited.  Those hypotheses are designed to force a Hamilton transversal and
then, in stronger forms, a rich cycle spectrum.  Here the Hamilton transversal
is part of the input.  We ask instead for the amount of \emph{aggregate}
system density that upgrades an already Hamiltonian graph system to a
pancyclic one.  The relevant quantity is
$
  \sum_{v\in V}\delta(v),
$
rather than a common minimum-degree or Ore-type condition on the members.
Different vertices may attain their minimum degree in different members, and
some may have system degree below $n/2$; in particular, our hypothesis need
not imply $\delta(G_i)\ge n/2$ for every $i$.  Thus the result is a
Hamiltonian-to-pancyclic statement in the sense of Bondy's density theorem,
rather than another criterion for transversal Hamiltonicity.  Combining our
theorem with the rainbow Dirac theorem of Joos and Kim recovers the exact
Dirac-type pancyclicity conclusion of Li, Li and Li.

A Hamilton transversal alone gives little control over shorter cycles: if
all members are the same $n$-cycle, the system has a rainbow Hamilton cycle
but no shorter cycle.  Moreover, in the rainbow setting a chord is useful
only when it can be assigned a color released by the Hamilton arc it
replaces.  This released-color constraint is the main additional feature
absent from the ordinary Bondy argument and motivates the shortcut counting
used below.

Our main result is the following Bondy-type theorem for graph systems.

\begin{theorem}[Rainbow Hamiltonian pancyclicity]\label{thm:bondy}
Let $n\ge3$ and let $\G=(G_1,\ldots,G_n)$ be a graph system on an
$n$-vertex set $V$.  Suppose that $\G$ contains a rainbow Hamilton cycle
and
\begin{equation}\label{eq:main-total}
  \sum_{v\in V}\delta(v)\ge
  \left\lceil\frac{n^2}{2}\right\rceil-1.
\end{equation}
Then $\G$ is rainbow pancyclic, unless $n$ is even and there is a
partition $V=X\sqcup Y$, with $|X|=|Y|=n/2$, such that
$G_i=K_{X,Y}$ for every $i\in[n]$.
\end{theorem}

For even $n=2m$ the threshold is exact at the integer level.  Take one
member to be $K_{m,m}-e$ and the remaining $n-1$ members to be
$K_{m,m}$.  The system has a rainbow Hamilton cycle and
$
  \sum_v\delta(v)=\frac{n^2}{2}-2
$. 
But, being bipartite, it has no odd cycle and is not the exceptional
system in Theorem~\ref{thm:bondy}.  Hence the theorem fails at the next
integer level below its stated bound.

\begin{remark}[The diagonal specialization]\label{rem:diagonal}
If $G_1=\cdots=G_n=G$, then \eqref{eq:main-total} becomes an ordinary edge
condition.  For even $n$ it is exactly Bondy's threshold
$e(G)\ge n^2/4$, while for odd $n$ it becomes
$e(G)\ge (n^2-1)/4=\lfloor n^2/4\rfloor$.  We record this only to locate
the aggregate condition relative to the classical theorem; it is not an
improvement of the best ordinary-graph density results.  In particular,
H\"aggkvist, Faudree and Schelp \cite{HaggkvistFaudreeSchelp1981} proved a
stronger density theorem for non-bipartite Hamiltonian graphs.  The novelty
here lies in allowing the minimizing member in $\delta(v)$ to vary with
$v$, a phenomenon absent in the diagonal setting.
\end{remark}

\paragraph{Idea of the proof.}
Fix a colored rainbow Hamilton cycle
\[
  C=v_0v_1\cdots v_{n-1}v_0, 
  \qquad  v_iv_{i+1}\in G_i \quad(i\in\Z_n),
\]
and suppose that a rainbow $C_{n-t}$ is missing.  A type-$a$ shortcut
replaces an $(a+1)$-edge Hamilton arc and deletes $a$ internal vertices.  We
count only shortcuts that occur in one of the two colors on the boundary
edges of the deleted arc; these colors are automatically released by the
replacement.

Endpoint crossing forces many such releasable shortcuts.  Conversely, a
type-$a$ shortcut and a type-$(t-a)$ shortcut cannot have edge-disjoint
Hamilton blocks, since the two replacements would produce a rainbow
$C_{n-t}$.  Thus complementary shortcut supports are cross-intersecting,
and a sharp cyclic-block bound gives the matching upper estimate.  For even
$t$ this creates a strict counting gap.  For odd $t$, equality forces one
support in each complementary pair to be empty and the other essentially
full, which yields an iterated distance-two switch.

The switch orbits are the decisive structural object.  In odd order they
connect all positions, contradicting the persistent forbidden-chord
condition.  In even order they are precisely the two parity classes, which
become independent in every member; the degree bound then forces the common
balanced complete bipartite graph.  At the lower integer threshold the same
mechanism survives because an exact defect identity leaves only one unit of
slack in odd order and two in even order.  The boundary lengths $C_3$ and
$C_{n-1}$ are handled separately.

Section~\ref{sec:baseline} proves the $n^2/2$ case,
Section~\ref{sec:boundary} treats the lower integer threshold, and
Section~\ref{sec:sharpness} gives the extremal examples.
\section{The \texorpdfstring{$n^2/2$}{n\string^2/2} regime: shortening and rigidity}
\label{sec:baseline}

We first prove the result under the stronger bound
$\sum_v\delta(v)\ge n^2/2$.  The proof has three ingredients: endpoint
crossing forces many releasable shortcuts; a missing cycle makes
complementary shortcut blocks cross-intersect; and the equality case yields
a local distance-two exchange whose position orbits determine the extremal
structure.  Section~\ref{sec:boundary} treats the remaining integer level.

\begin{proposition}[The $n^2/2$ case]
\label{prop:pancy-baseline}
Suppose that $\G$ contains a rainbow Hamilton cycle and
$
  \sum_{v\in V}\delta(v)\ge \frac{n^2}{2}.
$
Then $\G$ is rainbow pancyclic unless $n$ is even and all members are
the same balanced complete bipartite graph.
\end{proposition}

Throughout this section we assume the baseline inequality
\begin{equation}
  \sum_{v\in V}\delta(v)\ge \frac{n^2}{2}.
  \label{eq:pancy-global}
\end{equation}
Relabel the members of the system by $\Z_n$ and fix a rainbow Hamilton
cycle
\begin{equation}
  C=v_0v_1\cdots v_{n-1}v_0,
  \qquad
  v_iv_{i+1}\in E(G_i)
  \quad(i\in\Z_n).
  \label{eq:pancy-H}
\end{equation}
All subscripts in this section are read modulo $n$.  More generally,
let
$
  H=x_0x_1\cdots x_{n-1}x_0
$
is an indexed Hamilton cycle, write
\[
  e_i(H)=x_ix_{i+1},
  \qquad
  B_H(i,a)=\{e_i(H),e_{i+1}(H),\ldots,e_{i+a}(H)\}
\]
for the block of $a+1$ consecutive Hamilton edges starting at $x_i$.  The
associated Hamilton arc
$
  x_ix_{i+1}\cdots x_{i+a+1}
$
has length $a+1$ (number of edges) and exactly $a$ internal vertices.  For
the fixed cycle $C$ in \eqref{eq:pancy-H}, we abbreviate
$e_i=e_i(C)$.  Throughout the paper, \emph{size} refers to the cardinality
of a set of cyclic positions, while \emph{length} refers to the number of
edges of a graph path, cycle, or arc.

\subsection{Crossing, released colors, and local exchange}

The proof rests on four elementary tools: an endpoint crossing inequality,
an intersection bound for cyclic blocks, a lemma for simultaneous arc
replacements, and an exchange-closure statement for iterated local switches.

\begin{lemma}[Path crossing]
\label{lem:crossing}
Let
$
  P=x_1x_2\cdots x_\ell
$
be a rainbow path with $\ell\ge3$, and let $a\ne b$ be two colors unused on $P$.  If
$V(P)$ spans no rainbow $C_\ell$, then
\[
  d_{G_a}(x_1,V(P))+d_{G_b}(x_\ell,V(P))\le \ell-1.
\]
\end{lemma}

\begin{proof}
The edge $x_1x_\ell$ lies in neither $G_a$ nor $G_b$, since otherwise
$P$ can be closed to a rainbow $C_\ell$.  Put
\[
 A=\{j\in[\ell-1]:x_1x_{j+1}\in E(G_a)\},
 \qquad
 B=\{j\in[\ell-1]:x_jx_\ell\in E(G_b)\}.
\]
If $j\in A\cap B$, delete $x_jx_{j+1}$ and add
$x_1x_{j+1}$ and $x_jx_\ell$ in colors $a$ and $b$, respectively.  This
produces a rainbow $C_\ell$, a contradiction.  Hence $A\cap B=\varnothing$,
and therefore
\[
  d_{G_a}(x_1,V(P))+d_{G_b}(x_\ell,V(P))
  =|A|+|B|\le \ell-1.
\]
\end{proof}

For $s\in\Z_n$ and $1\le p\le n$, write
\[
  I_p(s)=\{s,s+1,\ldots,s+p-1\}\subseteq\Z_n.
\]
We call $I_p(s)$ a \emph{cyclic $p$-block}.  It is a set of exactly $p$
consecutive positions, so $|I_p(s)|=p$.

\begin{lemma}[Cross-intersecting cyclic blocks]
\label{lem:arc}
Let $\mathcal A$ be a nonempty family of cyclic $p$-blocks and
$\mathcal B$ a nonempty family of cyclic $q$-blocks.  Suppose that
$
  A\cap B$ for every $A\in\mathcal A,\ B\in\mathcal B,
$
and $p+q\le n-1$.  Then
$
  |\mathcal A|+|\mathcal B|\le p+q.
$
\end{lemma}

\begin{proof}
Fix a cyclic $p$-block $I_p(s)$.  A cyclic $q$-block is disjoint from it
exactly when its starting position belongs to
\[
  I_r(s+p)=\{s+p,s+p+1,\ldots,s+n-q\},
  \qquad r:=n-p-q+1.
\]
Thus the forbidden starting positions form a cyclic $r$-block.

Let $N$ be the set of starting positions of all cyclic $q$-blocks that are
disjoint from at least one member of $\mathcal A$.  Since every member of
$\mathcal B$ meets every member of $\mathcal A$, no member of
$\mathcal B$ starts in $N$.  Moreover, $\mathcal B\ne\varnothing$, so
$N\ne\Z_n$.

Choose a position outside $N$ and cut the cyclic order immediately before
that position.  Every forbidden cyclic $r$-block is contained in $N$, so
none crosses the cut.  Different members of $\mathcal A$ give different
forbidden starting positions, and hence different forbidden $r$-blocks.
After cutting, these are ordinary blocks of $r$ consecutive positions on
a line.  Order them by their left endpoints.  Since every block has size $r$,
their right endpoints are strictly increasing.  The first block contributes
$r$ positions and each later block contributes at least one new position.
Hence
$
  |N|\ge r+|\mathcal A|-1.
$
Therefore
$
  |\mathcal B|\le n-|N|
  \le n-r-|\mathcal A|+1.
$
Since $r=n-p-q+1$, we obtain
$
  |\mathcal A|+|\mathcal B|\le p+q.
$
\end{proof}

We also need the following elementary replacement lemma.

\begin{lemma}[Disjoint arc replacements]
\label{lem:disjoint-shortcuts}
Let
$
 H=x_0x_1\cdots x_{n-1}x_0
$
be a rainbow Hamilton cycle, with $x_jx_{j+1}$ colored $c_j$.  For $h=1,\ldots,s$, let
\[
  B_h=\{x_{p_h}x_{p_h+1},x_{p_h+1}x_{p_h+2},\ldots,
          x_{p_h+a_h}x_{p_h+a_h+1}\}
\]
be a block of $a_h+1$ consecutive Hamilton edges, where $a_h\ge1$.
Assume that the edge sets $B_1,\ldots,B_s$ are pairwise disjoint and that
$
  n-\sum_{h=1}^s a_h\ge3.
$
The associated Hamilton arcs may share endpoints.  Replace the arc
associated with $B_h$ by the chord
$
 f_h=x_{p_h}x_{p_h+a_h+1}.
$
Let $R$ be the set of colors on all deleted Hamilton edges.  If the chords
$f_1,\ldots,f_s$ can be assigned pairwise distinct legal colors from $R$,
then the resulting cycle is a rainbow cycle of length
$
 n-\sum_{h=1}^s a_h.
$
\end{lemma}

\begin{proof}
Because the Hamilton-edge blocks are pairwise edge-disjoint, the interiors of
the corresponding Hamilton arcs are disjoint; two consecutive arcs may
share an endpoint.  Delete the internal vertices of all chosen arcs and,
in the original cyclic order, join the two endpoints of each deleted arc
by its chord.  Exactly $n-\sum_h a_h$ distinct vertices remain, and the
hypothesis $n-\sum_h a_h\ge3$ ensures that they form one simple cycle.
The $h$th replacement deletes exactly $a_h$ internal vertices.  Every
retained Hamilton edge keeps its original color, which is outside $R$,
while the new chords receive pairwise distinct colors from $R$.  Hence the
resulting cycle is rainbow and has the stated length.
\end{proof}

\begin{remark}[The released-color constraint]
Lemma~\ref{lem:disjoint-shortcuts} isolates the genuinely rainbow part of
the shortening argument.  Geometrically, edge-disjoint replacement blocks
are enough to produce a shorter cycle.  In the rainbow setting, the new chords must also
receive distinct colors from the set released by the deleted Hamilton edges.
In general this is a small system-of-distinct-representatives problem.  Our
boundary-color convention is designed so that, for the one- and two-shortcut
replacements used below, the required representatives are supplied
automatically by the deleted blocks.
\end{remark}

\begin{lemma}[Exchange closure and the bipartite structure]
\label{lem:exchange-closure}
Assume $n\ge5$.  Let $\mathcal H$ be a nonempty collection of colored rainbow Hamilton
cycles of $\G$.  Suppose that, whenever
$
 H=x_0x_1\cdots x_{n-1}x_0\in\mathcal H,
$
the local switch
\begin{equation}\label{eq:abstract-switch}
 x_p\,x_{p+1}\,x_{p+2}\,x_{p+3}\,x_{p+4}
 \longmapsto
 x_p\,x_{p+3}\,x_{p+2}\,x_{p+1}\,x_{p+4}
\end{equation}
can be given a rainbow coloring that produces another member of
$\mathcal H$, for every $p\in\Z_n$.  Suppose further that for some even
integer $q$ with $2\le q\le n-2$ every $H\in\mathcal H$ satisfies
\begin{equation}\label{eq:abstract-forbidden}
 x_ix_{i+q}\notin E(G_r)
 \qquad\text{for all }i,r\in\Z_n.
\end{equation}
Then $n$ is even.  If $X$ and $Y$ are the vertices occupying the even
and odd positions, respectively, of any fixed $H\in\mathcal H$, then
$X$ and $Y$ are independent in every member $G_r$.  In particular, if
also
$
 \sum_{v\in V}\delta(v)\ge \frac{n^2}{2},
$
then $|X|=|Y|=n/2$ and
$
 G_r=K_{X,Y}$ for every $r\in\Z_n$.

\end{lemma}

\begin{proof}
The switch \eqref{eq:abstract-switch} exchanges the vertices in positions
$p+1$ and $p+3$ and fixes all other positions.  Thus it is an adjacent
transposition in the cyclic order obtained by repeatedly adding $2$ to a
position.

If $n$ is odd, addition by $2$ gives a single cyclic order on all $n$
positions.  Adjacent transpositions in this order generate all
permutations of the positions.  More concretely, one may first move a
chosen vertex to a prescribed position and then, after deleting that
fixed position from the cyclic order, move a second chosen vertex along
the remaining linear order.  Hence any prescribed pair of vertices can
be placed in positions $i$ and $i+q$.  By
\eqref{eq:abstract-forbidden} that pair is nonadjacent in every $G_r$.
Since the pair was arbitrary, every $G_r$ would be edgeless, contrary to
the existence of a Hamilton cycle in $\mathcal H$.  Therefore $n$ is
even.

For even $n$, addition by $2$ has exactly two cyclic orbits, the even
positions and the odd positions.  The allowed switches generate all
permutations inside each orbit.  Since $q$ is even, any two vertices that
start in the same parity class can therefore be placed in positions $i$
and $i+q$.  Equation \eqref{eq:abstract-forbidden} shows that every such
pair is nonadjacent in every $G_r$.  Thus both parity classes $X$ and $Y$
are independent in every $G_r$, and hence $G_r\subseteq K_{X,Y}$.

The two position classes have size $n/2$.  Consequently every vertex has
degree at most $n/2$ in every $G_r$, so $\delta(v)\le n/2$ for all $v$.
The total-degree hypothesis gives
$\sum_v\delta(v)\ge n^2/2$, which is exactly the sum of these $n$
pointwise upper bounds.  Hence $\delta(v)=n/2$ for every $v$.  Since
$d_{G_r}(v)\ge\delta(v)=n/2$ for every $v,r$, while
$G_r\subseteq K_{X,Y}$, every $G_r$ is exactly $K_{X,Y}$.
\end{proof}

\subsection{Shortcut supports and the complementary obstruction}

Fix $2\le t\le n-3$ and suppose that no rainbow $C_{n-t}$ exists.
The next notation separates the two sides of the argument.  The quantities
$X_a$ record how many \emph{released-color incidences} are forced by density,
while the support sets $S_a$ retain only the geometric positions of the
corresponding shortcuts.  The absence of $C_{n-t}$ will force the supports
for types $a$ and $t-a$ to cross-intersect.

For a type-$a$ shortcut, the two colors on the first and last deleted
Hamilton edges will be called its boundary colors.  It is useful to define
the support relative to the \emph{colored Hamilton cycle}, rather than only
to the fixed normalized cycle.  If
$
  H=x_0x_1\cdots x_{n-1}x_0$,
  $x_ix_{i+1}$ has color $c_i$.
Then for $1\le a\le t-1$, put
\begin{equation}
  S_a(H)=\{i\in\Z_n:
  x_ix_{i+a+1}\in E(G_{c_i})\cup E(G_{c_{i+a}})\}.
\end{equation}
Thus $i\in S_a(H)$ precisely when the type-$a$ chord replacing
$B_H(i,a)$ is available in at least one of the two colors released at the
ends of that deleted block.  Let $X_a(H)$ count the two boundary-color
incidences separately:
\begin{equation}
  X_a(H)=\sum_{i\in\Z_n}
  \left(
  \ind_{\{x_ix_{i+a+1}\in E(G_{c_i})\}}
  +
  \ind_{\{x_ix_{i+a+1}\in E(G_{c_{i+a}})\}}
  \right).
\end{equation}
For the normalized cycle $C$ in \eqref{eq:pancy-H}, we abbreviate
$
  S_a:=S_a(C)$, $X_a:=X_a(C)$.
Equivalently,
\begin{equation}
  S_a=\{i\in\Z_n:
  v_iv_{i+a+1}\in E(G_i)\cup E(G_{i+a})\},
  \label{eq:Sa}
\end{equation}
and
\begin{equation}
  X_a=\sum_{i\in\Z_n}
  \left(
  \ind_{\{v_iv_{i+a+1}\in E(G_i)\}}
  +
  \ind_{\{v_iv_{i+a+1}\in E(G_{i+a})\}}
  \right).
  \label{eq:Xa}
\end{equation}
In either notation,
\begin{equation}
  X_a(H)\le2|S_a(H)|
  \text{ and } X_a\le2|S_a|.
  \label{eq:X-support}
\end{equation}

\smallskip
\noindent\emph{Lower bound from crossing.}
For each
$i\in\Z_n$, consider the rainbow path
$
  P_i=v_iv_{i+1}\cdots v_{i-t-1},
$
which has $n-t$ vertices.  The colors $i-1$ and $i-t-1$ are unused on
$P_i$.  Since there is no rainbow $C_{n-t}$, Lemma~\ref{lem:crossing}
gives
\begin{equation}
  d_{G_{i-1}}(v_i,P_i)
  +d_{G_{i-t-1}}(v_{i-t-1},P_i)
  \le n-t-1.
  \label{eq:count-cross}
\end{equation}
The vertices outside $P_i$ are
$v_{i-t},v_{i-t+1},\ldots,v_{i-1}$.  We record the exterior incidences at
the two endpoints explicitly.  At $v_i$, the exterior vertices may be
written as
$
  v_{i-a-1}$ for $0\le a\le t-1.
$
For $a=0$ we obtain the Hamilton edge $v_iv_{i-1}$, which lies in
$G_{i-1}$.  For $1\le a\le t-1$, the incidence
$
  v_iv_{i-a-1}\in E(G_{i-1})
$
is exactly the second boundary-color incidence counted in $X_a$, at the
position $i-a-1$.

For the other endpoint put $j=i-t-1$.  Its exterior vertices are
$v_{j+a+1}$ for $0\le a\le t-1$.  Again $a=0$ gives the Hamilton edge
$v_jv_{j+1}\in E(G_j)$, while for $1\le a\le t-1$ the incidence
$
  v_jv_{j+a+1}\in E(G_j)
$
is exactly the first boundary-color incidence counted in $X_a$, at position
$j$.  As $i$ runs through $\Z_n$, every indicator occurring in every $X_a$
is obtained exactly once in this way.  Hence the total number of exterior
incidences in the two relevant colors is
$
  2n+\sum_{a=1}^{t-1}X_a.
$
Summing the two endpoint degrees before subtracting these exterior
incidences gives
\begin{align*}
&\sum_{i\in\Z_n}
\bigl(
  d_{G_{i-1}}(v_i,P_i)
  +d_{G_{i-t-1}}(v_{i-t-1},P_i)
\bigr)\\
&\qquad=
\sum_{i\in\Z_n}
\bigl(
  d_{G_{i-1}}(v_i)
  +d_{G_{i-t-1}}(v_{i-t-1})
\bigr)
-
\left(2n+\sum_{a=1}^{t-1}X_a\right)\\
&\qquad\ge
2\sum_{v\in V}\delta(v)-2n-
\sum_{a=1}^{t-1}X_a.
\end{align*}
Together with \eqref{eq:count-cross}, this yields
\[
  n(n-t-1)
  \ge 2\sum_{v\in V}\delta(v)-2n-
      \sum_{a=1}^{t-1}X_a.
\]
Using \eqref{eq:pancy-global}, we obtain
\begin{equation}
  \sum_{a=1}^{t-1}X_a\ge n(t-1).
\end{equation}
Using \eqref{eq:X-support}, we obtain
\begin{equation}
  \sum_{a=1}^{t-1}|S_a|\ge\frac{n(t-1)}2.
  \label{eq:Sa-lower}
\end{equation}

\smallskip
\noindent\emph{Complementary upper bound.}
Fix $1\le a\le t-1$.  If $i\in S_a$ and $j\in S_{t-a}$ had edge-disjoint Hamilton-edge
blocks
$
  B_C(i,a)$
  and
  $B_C(j,t-a)$,
choose for each chord a boundary color in which it is present.  Each chosen
color belongs to a Hamilton edge in its own deleted block, and the two
colors are distinct because the two blocks are edge-disjoint.  The
corresponding arcs may share an endpoint, which is allowed in
Lemma~\ref{lem:disjoint-shortcuts}.  Since the resulting cycle has length
at least $3$, that lemma gives a rainbow cycle of length
$
  n-a-(t-a)=n-t,
$
a contradiction.  Hence every Hamilton-edge block arising from $S_a$
meets every block arising from $S_{t-a}$.  Equivalently, the cyclic
$(a+1)$-blocks of edge indices from $S_a$ cross-intersect the cyclic
$(t-a+1)$-blocks from $S_{t-a}$.  If both support sets are nonempty,
Lemma~\ref{lem:arc} gives
\begin{equation}
  |S_a|+|S_{t-a}|\le t+2,
  \label{eq:pair-short}
\end{equation}
because the two cyclic blocks have sizes $a+1$ and $t-a+1$, and
$(a+1)+(t-a+1)=t+2\le n-1$.  In particular,
\begin{equation}
  |S_a|+|S_{t-a}|\le n
  \qquad(1\le a\le t-1),
  \label{eq:pair-n}
\end{equation}
where the inequality is immediate if one of the two sets is empty.

\subsection{Even shortening: a counting gap}

When $t$ is even, the complementary pairing has a middle,
self-complementary type.  This produces a strict gap between the density
lower bound and the cross-intersection upper bound.

\begin{lemma}[Even $t$]
\label{thm:even}
Let $t=2m$ with $2\le t\le n-3$.  Then $\G$ contains a rainbow
$C_{n-t}$.
\end{lemma}

\begin{proof}
Suppose not.  By \eqref{eq:pair-n}, 
$
  |S_a|+|S_{2m-a}|\le n
$ for $1\le a\le m-1$.
Hence these $2m-2$ sets contribute at most $n(m-1)$ in total.

It remains to bound $S_m$.  The claim $|S_m|\le m+1$ is trivial if
$S_m=\varnothing$.  Otherwise, two edge-disjoint blocks
$
  B_C(i,m)$ and  $B_C(j,m)\ (i,j\in S_m)$
would provide two shortcuts in distinct released boundary colors and hence,
by Lemma~\ref{lem:disjoint-shortcuts}, a rainbow $C_{n-2m}$.
Thus the cyclic $(m+1)$-blocks
$
  \mathcal I=\bigl\{I_{m+1}(i):i\in S_m\bigr\}
$
are pairwise intersecting.  Since $2(m+1)=t+2\le n-1$, Lemma~\ref{lem:arc} with
$\mathcal A=\mathcal B=\mathcal I$ gives
$
  2|S_m|=2|\mathcal I|\le2(m+1).
$
so $|S_m|\le m+1$.  Therefore
\[
  \sum_{a=1}^{2m-1}|S_a|
  \le n(m-1)+m+1.
\]
On the other hand, \eqref{eq:Sa-lower} gives
\[
  \sum_{a=1}^{2m-1}|S_a|\ge\frac{n(2m-1)}2=nm-\frac n2.
\]
Combining the two bounds gives
$
 nm-\frac n2\le n(m-1)+m+1$.
Thus $n\le2m+2=t+2$, contradicting $t\le n-3$.
\end{proof}

\subsection{Odd shortening: equality and parity rigidity}

When $t$ is odd, there is no self-complementary shortcut type: the types
pair perfectly as $(a,t-a)$.  The lower and upper support counts can
therefore coincide.  The task is no longer to obtain a counting
contradiction, but to understand the resulting equality structure and
propagate it by local exchange.

\begin{lemma}[Odd $t$]
\label{thm:odd}
Let $t$ be odd with $3\le t\le n-3$.  If there is no rainbow
$C_{n-t}$, then $n$ is even and there is a partition
$ V=X\sqcup Y$ with
  $|X|=|Y|=\frac n2,$
such that
\[
  G_0=\cdots=G_{n-1}=K_{X,Y}.
\]
\end{lemma}

\begin{proof}
\smallskip
\noindent\emph{Exact support structure.}
Write $t=2m+1$.  From \eqref{eq:pair-n},
\[
  \sum_{a=1}^{t-1}|S_a|
  =\sum_{a=1}^{m}\bigl(|S_a|+|S_{t-a}|\bigr)
  \le mn,
\]
whereas \eqref{eq:Sa-lower} gives the reverse inequality.  Hence the total is
$mn$.  Since each of the $m$ summands above is at most $n$, every one is
exactly $n$. For $1\le a\le m$ we have
\begin{equation}
  |S_a|+|S_{t-a}|=n.
\end{equation}
If both sets were nonempty, \eqref{eq:pair-short} would give at most
$t+2\le n-1$, a contradiction.  Hence one of $S_a,S_{t-a}$ is empty
and the other is all of $\Z_n$.
On the other hand,
\[
  n(t-1)
  \le\sum_{a=1}^{t-1}X_a
  \le2\sum_{a=1}^{t-1}|S_a|
  =n(t-1).
\]
Thus
$
 \sum_{a=1}^{t-1}\bigl(2|S_a|-X_a\bigr)=0.
$
Every summand is nonnegative, so
$
  X_a=2|S_a|$ for every $1\le a\le t-1.
$
Consequently, for each $1\le b\le t-1$, either $S_b=\varnothing$ and
no type-$b$ chord occurs in a boundary color, or $S_b=\Z_n$ and every
one of the $2n$ boundary incidences occurs. Equivalently,
\begin{equation}
  v_iv_{i+b+1}\in E(G_i)\cap E(G_{i+b})
  \qquad\forall i\in\Z_n.
\end{equation}
Exactly one of the two chord lengths $b$ and $t-b$ has the second
property.  More generally, the same dichotomy holds for $S_b(H)$ and
$S_{t-b}(H)$ on every colored rainbow Hamilton cycle $H$.  Indeed, globally
permuting the member labels so that the color of $x_ix_{i+1}$ becomes $i$
does not change any system degree or the existence of a rainbow cycle, and
the preceding count then applies verbatim to the cyclic order of $H$.

\smallskip
\noindent\emph{Creating full type-$2$ support.}
First choose a rainbow Hamilton cycle on which every chord joining vertices
three steps apart is present in both boundary colors.  Start with
any rainbow Hamilton cycle $H$ as above and choose
$b\in\{1,\ldots,t-1\}$ for which the second alternative holds.  Thus
\[
  x_ix_{i+b+1}\in E(G_{c_i})\cap E(G_{c_{i+b}})
  \qquad\forall i\in \Z_n.
\]
Replace the segment
$
  x_0x_1\cdots x_{b+1}x_{b+2}
$
by
$
  x_0x_{b+1}x_b\cdots x_1x_{b+2},
$
and color its edges, in order, by
$
  c_0,c_b,c_{b-1},\ldots,c_1,c_{b+1}.
$
The first new boundary edge $x_0x_{b+1}$ is available in color $c_0$, and
the second, $x_1x_{b+2}$, is available in color $c_{b+1}$; the middle
edges are old Hamilton edges traversed backwards.  The new segment thus
uses exactly the distinct colors $c_0,c_1,\ldots,c_{b+1}$, so the
resulting cycle is again a rainbow Hamilton cycle.

Write this new cycle as
\[
  H'=y_0y_1\cdots y_{n-1}y_0,
  \qquad y_iy_{i+1}\text{ colored }d_i,
\]
with the indices chosen so that
$
  y_{-1}=x_{-1}$, $y_0=x_0$,
  $y_1=x_{b+1}$ and $y_2=x_b$.
The edge $y_{-1}y_0=x_{-1}x_0$ is unchanged, so
$d_{-1}=c_{-1}$.  Moreover, applied at $i=-1$ to the second alternative for the original
cycle,  gives
\[
  y_{-1}y_2=x_{-1}x_b\in E(G_{c_{-1}})=E(G_{d_{-1}}).
\]
Hence on $H'$ at least one chord joining vertices three steps apart is
present in one of its two boundary colors.  Applying the conclusion above
to $H'$ with $b=2$, the alternative in which all such chords are absent is
impossible.  Hence every chord joining vertices three steps apart is
present in both boundary colors.  Renaming $H'$ as
\[
  H=x_0x_1\cdots x_{n-1}x_0,
  \qquad x_ix_{i+1}\text{ colored }c_i,
\]
we therefore have
\begin{equation}
  x_ix_{i+3}\in E(G_{c_i})\cap E(G_{c_{i+2}})
  \qquad\forall i\in \Z_n.
  \label{eq:odd-step3}
\end{equation}

\smallskip
\noindent\emph{Closure under the local switch.}
Equation \eqref{eq:odd-step3} allows the following distance-two exchange at every
position $p$:
\begin{equation}
  x_p\,x_{p+1}\,x_{p+2}\,x_{p+3}\,x_{p+4}
  \quad\longmapsto\quad
  x_p\,x_{p+3}\,x_{p+2}\,x_{p+1}\,x_{p+4}.
  \label{eq:odd-switch}
\end{equation}
Indeed, use color $c_p$ on $x_px_{p+3}$ and color $c_{p+3}$ on
$x_{p+1}x_{p+4}$; these choices are legal by
\eqref{eq:odd-step3} applied at $p$ and $p+1$, respectively.  Keep colors
$c_{p+2}$ and $c_{p+1}$ on the two middle edges, now traversed backwards.
The four colors are distinct, so the result is another rainbow Hamilton
cycle, and the vertices in positions $p+1$ and $p+3$ have been exchanged.

We next show that the same exchange remains available after it is performed.
Let the new Hamilton cycle be
\[
  H'=y_0y_1\cdots y_{n-1}y_0,
  \qquad y_iy_{i+1}\text{ colored }d_i,
\]
where outside the displayed five-vertex segment the positions are
unchanged.  In particular,
$
  y_{p-1}=x_{p-1}$, $y_p=x_p$,
  $y_{p+1}=x_{p+3}$, $y_{p+2}=x_{p+2}$,
and the unchanged edge $y_{p-1}y_p$ has color
$d_{p-1}=c_{p-1}$.  Before the exchange, \eqref{eq:odd-step3} gives
\[
  y_{p-1}y_{p+2}=x_{p-1}x_{p+2}\in E(G_{c_{p-1}})
  =E(G_{d_{p-1}}).
\]
Hence $H'$ again has a chord joining vertices three steps apart in a
boundary color.  Applying the conclusion preceding the construction to
$H'$ with $b=2$ again rules out the empty alternative and recovers
\eqref{eq:odd-step3} with the new vertices and colors.  Consequently the
exchange \eqref{eq:odd-switch} may be iterated.

\smallskip
\noindent\emph{Excluding complementary chords.}
We claim that on every Hamilton cycle obtained in this way,
\begin{equation}
  x_ix_{i+t-1}\notin E(G_r)
  \qquad\forall i,r\in\Z_n.
  \label{eq:odd-forbidden}
\end{equation}
Since all chords joining vertices three steps apart occur in both boundary
colors, the complementary chords that delete $t-2$ vertices occur in
neither boundary color.  Suppose nevertheless that
$x_ix_{i+t-1}\in E(G_r)$ for some $i,r$, and let $x_qx_{q+1}$ be the
Hamilton edge of color $r$.  The chord $x_ix_{i+t-1}$ replaces the Hamilton-edge block
$
  B=B_H(i,t-2),
$
which contains $t-1$ edges.  Its complement is a Hamilton arc of
$n-t+1\ge4$ edges.  Choose a block $J$ of three consecutive Hamilton
edges entirely inside this complementary arc, so that $B$ and $J$ are
edge-disjoint.  If $e_q(H)\notin B$, choose $J$ to contain $e_q(H)$. This is
possible because the complementary arc has at least four edges.  By \eqref{eq:odd-step3}, the chord replacing $J$ is present in
both boundary colors of $J$.

Assign color $r$ to $x_ix_{i+t-1}$.  If $e_q(H)\in B$, color $r$ is the color of a deleted edge of $B$; if
$e_q(H)\notin B$, it is the color of a deleted edge of $J$.  Thus $r$ belongs
to the released-color set of $B\cup J$.
The two boundary colors of $J$ are distinct colors of deleted edges of
$J$, so at least one of them is different from $r$.  Use such a color on
the chord replacing $J$.  The two new chords now have distinct legal
colors, both released by the two edge-disjoint Hamilton-edge blocks $B$
and $J$.
Since $n-(t-2)-2=n-t\ge3$, Lemma~\ref{lem:disjoint-shortcuts} gives a
rainbow cycle
of length
$
  n-(t-2)-2=n-t,
$
a contradiction.  This proves \eqref{eq:odd-forbidden}.

Let $\mathcal H_t$ be the collection of colored rainbow Hamilton cycles
obtainable from the present cycle by a finite sequence of switches
\eqref{eq:odd-switch}.  The preceding closure argument shows that
$\mathcal H_t$ is closed under every switch
\eqref{eq:abstract-switch}, and \eqref{eq:odd-forbidden} holds on every
member of $\mathcal H_t$.  Since $q=t-1$ is even and
$2\le q\le n-4$, Lemma~\ref{lem:exchange-closure} together with
\eqref{eq:pancy-global} gives an even order $n$ and a balanced
partition $V=X\sqcup Y$ such that
$
  G_r=K_{X,Y}$ for all $r$.
\end{proof}

\subsection{The boundary length \texorpdfstring{$C_{n-1}$}{C(n-1)}}

\begin{lemma}[Almost-spanning obstruction]
\label{thm:t1}
Let $n\ge5$.  If there is no rainbow $C_{n-1}$, then $n$ is even and all
members of the system are the same balanced complete bipartite graph.
\end{lemma}

\begin{proof}
\smallskip
\noindent\emph{Exact crossing rows.}
Keep the Hamilton cycle \eqref{eq:pancy-H}.  Fix $k\in\Z_n$ and delete
$v_k$.  Write the remaining Hamilton path as
$
  P_k=x_1\cdots x_{n-1}$ with
   $x_j=v_{k+j}$.
Its unused colors are $k$ and $k-1$, so Lemma~\ref{lem:crossing} gives
\begin{equation}
 d_{G_k}(v_{k+1},V\setminus\{v_k\})
 +d_{G_{k-1}}(v_{k-1},V\setminus\{v_k\})
 \le n-2.
 \label{eq:t1-cross}
\end{equation}
The only vertex outside $P_k$ is $v_k$, and both endpoint--$v_k$ edges
are Hamilton edges in the relevant colors.  Therefore
\[
\begin{aligned}
  \delta(v_{k+1})+\delta(v_{k-1})
  &\le d_{G_k}(v_{k+1})+d_{G_{k-1}}(v_{k-1})\\
  &=d_{G_k}(v_{k+1},P_k)+d_{G_{k-1}}(v_{k-1},P_k)+2\le n.
\end{aligned}
\]
Summing the last display over $k$ gives
$
 2\sum_{v\in V}\delta(v)\le n^2.
$
Together with \eqref{eq:pancy-global}, equality holds.  Since each of the
$n$ displayed inequalities has left-hand side at most $n$, each one is tight.
Moreover, equality in its two-step chain forces both the layer-degree
comparison and \eqref{eq:t1-cross} to be tight for every $k$.

As in the proof of Lemma~\ref{lem:crossing}, put
\[
 A_k=\{j\in[n-2]:x_1x_{j+1}\in E(G_k)\},
 \qquad
 B_k=\{j\in[n-2]:x_jx_{n-1}\in E(G_{k-1})\}.
\]
The crossing argument gives $A_k\cap B_k=\varnothing$, while equality in
\eqref{eq:t1-cross} gives
\begin{equation}
  A_k\sqcup B_k=[n-2].
  \label{eq:t1-complement}
\end{equation}
The closing edge $x_1x_{n-1}$ belongs to neither unused color, so
$1\notin B_k$ and $1\in A_k$.  Reindexing gives
\begin{equation}
  v_iv_{i+1}\in E(G_{i-1})\cap E(G_i)
  \qquad\forall i.
  \label{eq:doubleH}
\end{equation}

\smallskip
\noindent\emph{Distance two is forbidden; distance three is forced.}
We claim that
\begin{equation}
  v_iv_{i+2}\notin E(G_r)
  \qquad\forall i,r.
  \label{eq:t1-step2}
\end{equation}
Suppose otherwise that $v_iv_{i+2}\in E(G_r)$.  Write
$e_j=v_jv_{j+1}$, so that $e_j$ has color $j$.  Delete $e_i$ and
$e_{i+1}$, the remaining Hamilton edges together with the chord
$v_iv_{i+2}$ form a cycle on $n-1$ vertices.  If
$r\in\{i,i+1\}$, color the chord with $r$ and keep every remaining
Hamilton edge in its original color. Then we  obtain a rainbow $C_{n-1}$.

Assume therefore that $r\notin\{i,i+1\}$.  The edge $e_r$ lies on the
retained Hamilton arc from $v_{i+2}$ to $v_i$.  Starting with $e_{i+2}$
and proceeding along this arc through $e_r$, recolor every encountered
edge $e_j$ with color $j-1$. Here the indices are read in this forward
cyclic order.  This is legal by \eqref{eq:doubleH}, since
$
  e_j=v_jv_{j+1}\in E(G_{j-1})\cap E(G_j).
$
The first new color, $i+1$, was released when $e_{i+1}$ was deleted.  For
each subsequent edge $e_j$, its new color $j-1$ was the old color of the
preceding edge $e_{j-1}$, which has already been recolored.  Thus the
colors on all retained edges remain distinct.  Finally, $e_r$ no longer
uses its old color $r$, so color $r$ is free.  Assigning color $r$ to
$v_iv_{i+2}$ gives a rainbow $C_{n-1}$, again a contradiction.  Hence
\eqref{eq:t1-step2} holds.

If $n=5$, then $[n-2]=[3]$.  At $j=2$, both crossing candidates
$x_1x_3$ and $x_2x_4$ are forbidden by \eqref{eq:t1-step2}.  Thus
$2\notin A_k\cup B_k$, contradicting \eqref{eq:t1-complement}.  Hence the
assumed obstruction does not occur when $n=5$, and from now on $n\ge6$.
At $j=2$, the edge $x_1x_3$ is forbidden by \eqref{eq:t1-step2}, so
$2\in B_k$.  Hence
$
  x_2x_{n-1}=v_{k+2}v_{k-1}\in E(G_{k-1}).
$
Putting $i=k-1$ gives $v_iv_{i+3}\in E(G_i)$.  At $j=n-3$, the edge
$x_{n-3}x_{n-1}$ is forbidden, so $n-3\in A_k$.  Hence
$
  x_1x_{n-2}=v_{k+1}v_{k-2}\in E(G_k).
$
Putting $i=k-2$ gives $v_iv_{i+3}\in E(G_{i+2})$.  Since $k$ is
arbitrary, the two conclusions together yield
\begin{equation}
  v_iv_{i+3}\in E(G_i)\cap E(G_{i+2})
  \qquad\forall i.
\end{equation}

\smallskip
\noindent\emph{Exchange closure.}
These conclusions depend only on the current colored Hamilton cycle and the absence of a rainbow $C_{n-1}$.  Hence, for every rainbow Hamilton cycle
$H=x_0x_1\cdots x_{n-1}x_0$,
  $x_jx_{j+1}$ colored $c_j$,
we have
\begin{align}
  x_jx_{j+2}&\notin E(G_r)
  &&\forall j,r,
  \label{eq:t1-step2-arbitrary}\\
  x_jx_{j+3}&\in E(G_{c_j})\cap E(G_{c_{j+2}})
  &&\forall j.
  \label{eq:t1-step3-arbitrary}
\end{align}
Thus, for every $p$, the replacement
\[
  x_p\,x_{p+1}\,x_{p+2}\,x_{p+3}\,x_{p+4}
  \quad\longmapsto
  x_p\,x_{p+3}\,x_{p+2}\,x_{p+1}\,x_{p+4}
\]
gives another rainbow Hamilton cycle: use colors $c_p$ and $c_{p+3}$ on
the two new boundary edges and keep colors $c_{p+2}$ and $c_{p+1}$ on
the reversed middle edges.  Since
\eqref{eq:t1-step2-arbitrary}--\eqref{eq:t1-step3-arbitrary} hold for every
rainbow Hamilton cycle, this exchange of the vertices in positions
$p+1$ and $p+3$ may be repeated.

Let $\mathcal H_1$ be the collection of colored rainbow Hamilton cycles
obtainable from the fixed Hamilton cycle by these switches.  It is closed
under \eqref{eq:abstract-switch}, and
\eqref{eq:t1-step2-arbitrary} gives
\eqref{eq:abstract-forbidden} with $q=2$ on every member of
$\mathcal H_1$.  Lemma~\ref{lem:exchange-closure}, together with
\eqref{eq:pancy-global}, therefore yields an even order $n$ and a
balanced partition $V=X\sqcup Y$ for which
$
  G_r=K_{X,Y}$ for all $r$.
\end{proof}

\subsection{Completion of the baseline argument}

\begin{proof}[Proof of Proposition~\ref{prop:pancy-baseline}]
We first handle $n\le5$.  For $n=3$, the Hamilton cycle is already a
rainbow $C_3$.  Let $n=4$ and write it as
\[
  v_0v_1v_2v_3v_0,
  \qquad v_iv_{i+1}\in E(G_i).
\]
If there is no rainbow triangle, neither diagonal belongs to any member of
the system.  Indeed, suppose for example that $v_0v_2\in E(G_r)$.  The
two Hamilton $v_0$--$v_2$ arcs use the disjoint color pairs $\{0,1\}$ and
$\{2,3\}$; one of these pairs avoids $r$, and together with $v_0v_2$
gives a rainbow triangle.  Hence every $G_r$ is a subgraph of
$
  K_{\{v_0,v_2\},\{v_1,v_3\}}.
$
Thus every degree in every $G_r$ is at most $2$, and hence
$\delta(v)\le2$ for every $v$.  Since the total degree sum is at least
$8$, we must have $\delta(v)=2$ for every $v$.  Therefore
$d_{G_r}(v)\ge2$ for every $v,r$, while the bipartite containment gives
the reverse inequality.  Hence every $G_r$ is this $K_{2,2}$.
For $n=5$, Lemma~\ref{thm:even} with $t=2$ gives a rainbow $C_3$, while
Lemma~\ref{thm:t1} shows that a rainbow $C_4$ cannot be missing.  Together
with the Hamilton $C_5$, the system is rainbow pancyclic.

Assume now that $n\ge6$.  The Hamilton cycle gives a rainbow $C_n$.  If a
rainbow $C_{n-1}$ is missing, Lemma~\ref{thm:t1} gives the common
balanced complete bipartite system.
Let $3\le\ell\le n-2$ and put $t=n-\ell$.  Then $2\le t\le n-3$.  If
$t$ is even, Lemma~\ref{thm:even} gives a rainbow $C_\ell$.  If $t$ is
odd, Lemma~\ref{thm:odd} shows that the absence of a rainbow $C_\ell$
forces the common balanced complete bipartite system.

Thus every missing length forces the same common balanced bipartite
system.  Otherwise all lengths $\ell\in\{3,\ldots,n\}$ occur and the
system is rainbow pancyclic.  The two alternatives are disjoint because the
common bipartite system has no odd cycle.
\end{proof}

\section{Stability at the lower integer threshold}
\label{sec:boundary}

We now assume equality at the lower integer level.  The counting argument
from Section~\ref{sec:baseline} remains valid, but with total slack equal to
one when $n$ is odd and two when $n$ is even.  We record this slack explicitly
and show that the distance-two switching argument still applies.  Throughout
this section a rainbow Hamilton cycle is fixed and
\begin{equation}
  \sum_{v\in V}\delta(v)
  =B_n:=\left\lceil\frac{n^2}{2}\right\rceil-1.
  \label{eq:boundary-layer}
\end{equation}
Equivalently,
\begin{equation}
 n^2-2B_n=
 \begin{cases}
  1,& n\text{ odd},\\
  2,& n\text{ even}.
 \end{cases}
 \label{eq:boundary-gap}
\end{equation}
For any fixed colored rainbow Hamilton cycle, we may apply a global
permutation to the member labels and normalize it as
\[
 H=x_0x_1\cdots x_{n-1}x_0,
 \qquad x_ix_{i+1}\in E(G_i).
\]
Such a permutation leaves every system degree unchanged.  Whenever the
defect identity below is applied to a later Hamilton cycle, all row defects,
shortcut supports, incidence counts, and defect terms are recomputed after
this normalization.  Thus the total defect budget is invariant, although
the locations of the defects may move under switching.

\subsection{The defect identity}

For the crossing row associated with
$
 P_i=x_ix_{i+1}\cdots x_{i-t-1},
$
let
\[
 q_i=(n-t-1)-
 \bigl(d_{G_{i-1}}(x_i,P_i)+
       d_{G_{i-t-1}}(x_{i-t-1},P_i)\bigr)\ge0.
\]
Thus $q_i$ measures the deficit in the $i$th endpoint-crossing row.
When $t=2m+1$ is odd, four nonnegative quantities account for all slack:
\begin{align*}
 &D_{\rm deg}
 :=\sum_i\bigl(d_{G_{i-1}}(x_i)-\delta(x_i)
                 +d_{G_i}(x_i)-\delta(x_i)\bigr),\quad
 &&D_{\rm cross}:=\sum_iq_i,\\
 &D_{\rm supp}:=mn-\sum_{a=1}^{2m}|S_a|,
 &&D_{\rm col}:=2\sum_{a=1}^{2m}|S_a|-\sum_{a=1}^{2m}X_a.
\end{align*}
Here $D_{\rm supp}$ measures the loss from the maximal complementary-support
total $n$ for each pair $(a,t-a)$, while $D_{\rm col}$ counts missing
boundary-color incidences at supported positions.  A one-unit loss in
complementary support contributes two units to the identity below.

\begin{lemma}[Shortcut count and defect identity]
\label{lem:lower-shortcut-count}
Fix $2\le t\le n-3$ and suppose that there is no rainbow
$C_{n-t}$.  Define $S_a$ and $X_a$ as in \eqref{eq:Sa} and
\eqref{eq:Xa}.  Then
\begin{equation}
 \sum_{a=1}^{t-1}X_a
 \ge n(t-1)-(n^2-2B_n).
 \label{eq:lower-shortcut-count}
\end{equation}
Equality in \eqref{eq:lower-shortcut-count} can occur only when every
crossing row is tight and the two layer degrees used at each vertex are
equal to its system degree.
If $t=2m+1$ is odd, then the slack decomposes exactly as
\begin{equation}
 D_{\rm deg}+D_{\rm cross}+2D_{\rm supp}+D_{\rm col}
 =n^2-2B_n.
 \label{eq:defect-identity}
\end{equation}
In particular,
\begin{equation}
 2D_{\rm supp}+D_{\rm col}+D_{\rm cross}
 \le n^2-2B_n.
 \label{eq:odd-lower-count}
\end{equation}
Finally, if $q_i=0$, then
\begin{equation}
 x_ix_{i+1}\in E(G_{i-1}),
 \qquad
 x_{i-t-2}x_{i-t-1}\in E(G_{i-t-1}),
 \label{eq:qzero-end-edges}
\end{equation}
and hence, for every $j$,
\begin{equation}
 q_j=0\Longrightarrow x_jx_{j+1}\in E(G_{j-1}),
 \qquad
 q_{j+t+2}=0\Longrightarrow x_jx_{j+1}\in E(G_{j+1}).
 \label{eq:qzero-adjacent-colors}
\end{equation}
\end{lemma}

\begin{proof}
Before replacing layer degrees by system degrees, the counting argument
from Section~\ref{sec:baseline} gives the exact equality
\begin{align*}
 \sum_{a=1}^{t-1}X_a
 =n(t-1)-(n^2-2B_n)+D_{\rm deg}+D_{\rm cross},
\end{align*}
where the last two terms are nonnegative.  This proves
\eqref{eq:lower-shortcut-count} and its equality statement.

Now let $t=2m+1$.  Since
$
 2mn= n(t-1),
$
rearranging the preceding equality yields
\[
 n^2-2B_n
 =D_{\rm deg}+D_{\rm cross}
  +2\left(mn-\sum_{a=1}^{2m}|S_a|\right)
  +\left(2\sum_{a=1}^{2m}|S_a|-\sum_{a=1}^{2m}X_a\right),
\]
which is exactly \eqref{eq:defect-identity}.  The quantities
$D_{\rm supp}$ and $D_{\rm col}$ are nonnegative by
\eqref{eq:pair-n} and \eqref{eq:X-support}, respectively.  Dropping
$D_{\rm deg}$ gives \eqref{eq:odd-lower-count}.

If $q_i=0$, equality holds in the Path Crossing Lemma for $P_i$.  Its
two crossing sets partition all admissible positions.  The first and
last positions then force the two end Hamilton edges into the opposite
unused colors, giving \eqref{eq:qzero-end-edges}.  Reindexing gives
\eqref{eq:qzero-adjacent-colors}.
\end{proof}

For odd $t$, \eqref{eq:defect-identity} is the stability form of the exact
support equality from Section~\ref{sec:baseline}.  The total defect budget is
$1$ when $n$ is odd and $2$ when $n$ is even.  Thus a one-unit loss from
maximal complementary support already consumes two units, while a missing
boundary-color incidence or one unit of crossing deficit consumes one.  The later switching argument uses
only this budget accounting, together with the fact that the degree defect
$D_{\rm deg}$ can only reduce the room available for all other failures.

\subsection{The boundary length \texorpdfstring{$C_3$}{C3}}

Triangles are most efficiently handled through the following weighted
form of Mantel's theorem.

\begin{lemma}[Four-layer weighted Mantel lemma]
\label{lem:four-layer-mantel}
Let $F_1,F_2,F_3,F_4$ be graphs on the same $n$-vertex set.  If there is
no triangle whose three edges can be assigned three distinct labels from
$\{1,2,3,4\}$, then
\[
 \sum_{j=1}^4 e(F_j)\le4\left\lfloor\frac{n^2}{4}\right\rfloor.
\]
Equality holds if and only if there is a balanced-as-possible bipartition
$V=A\sqcup B$ such that
$
 F_1=F_2=F_3=F_4=K_{A,B}.
$
\end{lemma}

\begin{proof}
For each pair $e$ put
\[
 w(e)=|\{j:e\in E(F_j)\}|\in\{0,1,2,3,4\}.
\]
If a triangle has edge weights  $a\le b\le c$
with $a\ge1$, $b\ge2$, and $c\ge3$, then the three actual color sets
have cardinalities at least $1,2,3$.  Hall's theorem therefore gives
three distinct representatives, contrary to the hypothesis.  Thus no
triangle has such a weight pattern.

We first prove an auxiliary bound.  If all edge weights are at most
$3$ and $n\ge3$, then
\begin{equation}
 W\le n(n-1).
 \label{eq:max-three-bound}
\end{equation}
For $n=3,4$ the bound is explicit.  If all weights are at most
$2$, then $W\le2\binom n2=n(n-1)$.  Otherwise choose an edge $xy$ of
weight $3$.  For every other vertex $z$ the forbidden-pattern condition
forces $w(xz)+w(yz)\le3$. If $n=4$, then the only edge disjoint from $xy$
also has weight at most $3$.  Thus $W\le3+3=6$ for $n=3$ and
$W\le3+2\cdot3+3=12$ for $n=4$.  For $n\ge5$, argue by induction.  Again the case of maximum weight at most
$2$ is immediate.
If $w(xy)=3$, then
$
 w(xz)+w(yz)\le3\ (z\ne x,y),
$
and the induction hypothesis on $V\setminus\{x,y\}$ gives
\[
 W\le (n-2)(n-3)+3(n-2)+3
   =n^2-2n+3\le n(n-1).
\]

We prove by induction on $n$ that
\begin{equation}
 W:=\sum_e w(e)\le4\left\lfloor\frac{n^2}{4}\right\rfloor.
 \label{eq:weighted-mantel}
\end{equation}
For $n\le2$, both the inequality and the equality description are immediate.
For $n\ge3$, \eqref{eq:max-three-bound} shows that equality in
\eqref{eq:weighted-mantel} is impossible when the maximum edge weight is
at most $3$.  Hence, for the remaining case and for every equality case,
choose an edge $xy$ of weight $4$.

For such an edge, the same forbidden-pattern condition gives
$
 w(xz)+w(yz)\le4$\
 $(z\ne x,y).$
Therefore
\begin{align*}
 W
 &\le4\left\lfloor\frac{(n-2)^2}{4}\right\rfloor
    +4(n-2)+4
  =4\left\lfloor\frac{n^2}{4}\right\rfloor.
\end{align*}
If equality holds, equality holds on $V\setminus\{x,y\}$ and
$w(xz)+w(yz)=4$ for every $z$.  If both terms were positive, then the forbidden-pattern condition with
$w(xy)=4$ would force both of them to equal $1$, contradicting that
their sum is $4$.  Hence equality forces
\[
 \{w(xz),w(yz)\}=\{4,0\}.
\]
By induction, the remaining vertices have a balanced-as-possible
bipartition $A'\sqcup B'$ with weight $4$ exactly on its cross edges.
The weight-$4$ neighbors of $x$ cannot meet both $A'$ and $B'$, for
otherwise they form with $x$ a triangle of three weight-$4$ edges; the
same is true for $y$.  Since for every remaining vertex $z$ exactly one
of $xz,yz$ has weight $4$, the two weight-$4$ neighborhoods partition
$V\setminus\{x,y\}$.  Each is contained in one side of $A'\sqcup B'$, so
complementarity forces them to be the two whole sides (with the evident
interpretation when one side is empty).  Hence, after possibly
interchanging $A'$ and $B'$, $x$ is joined with weight $4$ to all of
$B'$ and to none of $A'$, while $y$ is joined with weight $4$ to all of
$A'$ and to none of $B'$.  Together with $w(xy)=4$, this gives the balanced bipartition
$(A'\cup\{x\})\sqcup(B'\cup\{y\})$ (up to interchanging $x$ and
$y$).  Thus the weight-$4$ graph is $T_2(n)$.  Its
$\lfloor n^2/4\rfloor$ cross edges already contribute
$4\lfloor n^2/4\rfloor=W$, so every remaining pair has weight $0$.
Hence each $F_j$ is exactly this same copy of $T_2(n)$, proving the
equality statement.
\end{proof}

\begin{lemma}[Rainbow triangle]
\label{lem:boundary-triangle}
Under the standing assumptions of this section, the system contains a rainbow triangle.
\end{lemma}

\begin{proof}
The case $n=3$ is the Hamilton cycle itself.  Assume $n\ge4$.  For every
$r$,
$
 2e(G_r)=\sum_v d_{G_r}(v)\ge B_n.
$
If $n$ is even, $B_n=n^2/2-1$ is odd, so parity gives
$e(G_r)\ge n^2/4$.  If $n$ is odd, then
$B_n=(n^2-1)/2$ and hence
$e(G_r)\ge(n^2-1)/4=\lfloor n^2/4\rfloor$.  Thus in all cases
$
 e(G_r)\ge\left\lfloor\frac{n^2}{4}\right\rfloor.
$
If no rainbow triangle exists, apply
Lemma~\ref{lem:four-layer-mantel} to any four members.  Equality is
forced, so those four members are the same $T_2(n)$.  For $n\ge5$,
using overlapping four-tuples shows that every member is the same
$T_2(n)$; for $n=4$ there is only one four-tuple.  If $n$ is odd this
common bipartite graph has no Hamilton cycle.  If $n$ is even it has
system-degree sum $n^2/2$, not $B_n$.  Both alternatives contradict the
standing assumptions.
\end{proof}

\subsection{Even shortening: the counting gap survives}

For even $t$, the support squeeze leaves only one possible obstruction:
a missing $C_4$.  In that extremal case every inequality in the count is
forced to be an equality.

\begin{lemma}[Even $t$ at the lower threshold]
\label{lem:boundary-even-t}
Let $2\le t\le n-3$ be even.  Then the system contains a rainbow
$C_{n-t}$.
\end{lemma}

\begin{proof}
Suppose not and write $t=2m$.  The support upper bound from the proof of
Lemma~\ref{thm:even} remains valid:
\[
 \sum_{a=1}^{2m-1}|S_a|\le n(m-1)+m+1.
\]
By Lemma~\ref{lem:lower-shortcut-count},
\begin{align*}
 2\sum_{a=1}^{2m-1}|S_a|
 \ge\sum_{a=1}^{2m-1}X_a
 \ge n(2m-1)-(n^2-2B_n).
\end{align*}
Consequently, 
$
 \ell\le2+(n^2-2B_n)
$ for $\ell=n-t$.
If $n$ is odd, then $n^2-2B_n=1$ and $\ell$ is odd, so $\ell=3$,
contrary to Lemma~\ref{lem:boundary-triangle}.  Hence $n$ is even,
$n^2-2B_n=2$, and $\ell$ is even, so $\ell=4$.  For this value of
$\ell$, the lower and upper bounds on $2\sum_a|S_a|$ coincide.  Hence every
inequality used between them is tight.  Lemma~\ref{lem:lower-shortcut-count}
then forces every crossing row to be tight and the relevant layer degrees to
equal the system degrees, also
$\sum_a(2|S_a|-X_a)=0$. So every supported shortcut occurs in both boundary
colors.  Finally, the upper bound
\[
 \sum_{a=1}^{2m-1}|S_a|
 \le \sum_{a=1}^{m-1} n +(m+1)
\]
is tight.  Since each complementary pair contributes at most $n$ and the
central support contributes at most $m+1$, all these component bounds are
tight.  In particular
\begin{equation}
 q_i=0\quad\text{for every }i,
 \qquad
 X_a=2|S_a|\quad(1\le a\le t-1),
 \label{eq:C4-all-equality}
\end{equation}
and for every $1\le a<m$,
\begin{equation}
 |S_a|+|S_{t-a}|=n.
 \label{eq:C4-pair-equality}
\end{equation}
Since two nonempty complementary supports would have total size at most
$t+2=n-2$, \eqref{eq:C4-pair-equality} gives
$
 \{S_a,S_{t-a}\}=\{\mathbb Z_n,\varnothing\}.
$
Moreover, \eqref{eq:C4-all-equality} says that every chord in a full
support occurs in both boundary colors.

We first dispose of the two small cases.  If $n=6$, then $t=2$ and
$|S_1|=2$.  As in the central-support argument of Lemma~\ref{thm:even},
the two corresponding cyclic $2$-edge blocks are pairwise intersecting.
Hence their starting positions are consecutive and, after a cyclic shift,
$S_1=\{0,1\}$.  Therefore,
$
 x_0x_2\in G_0\cap G_1$ and
 $x_1x_3\in G_1\cap G_2$.
Since every crossing row is tight,
\[
 x_ix_{i+1}\in G_{i-1}\cap G_i\cap G_{i+1}
 \qquad\forall i.
\]
Thus $x_0x_2x_3x_1x_0$ has the rainbow coloring $0,3,2,5$, a
contradiction.

If $n=8$, then $t=4$ and equality in the central-support bound gives
$|S_2|=3$.  Choose $i\in S_2$.  Then
$ x_ix_{i+3}\in G_i\cap G_{i+2}$ and 
 $
 x_{i+2}x_{i+3}\in G_{i+3}$,
where the second inclusion follows from $q_j=0$ for all $j$.  Hence
$x_ix_{i+1}x_{i+2}x_{i+3}x_i$ has the rainbow coloring
$i,i+1,i+3,i+2$, again a contradiction.

Assume now that $n\ge10$.  Then $t=n-4\ge6$, so $1,2<t/2$.  In
particular, all shortcut types used below lie in the range
$1\le b\le t-1$, and every two-block replacement still leaves the target
cycle length $n-t=4$.  The preceding equality argument depends only on the
current colored rainbow Hamilton cycle, and therefore remains valid after
any rainbow Hamilton switch.

\smallskip
\noindent\emph{Exchange principle.}
We claim that $S_b(H)\ne\varnothing$ is impossible for every colored rainbow
Hamilton cycle $H$ and every $b\in\{1,2\}$.  Suppose otherwise and apply the
equality argument to $H$.  Since $b<t/2$, it gives
$
 S_b(H)=\mathbb Z_n$,
 $
 S_{t-b}(H)=\varnothing
$,
and every type-$b$ chord occurs in both boundary colors.

Write the current Hamilton cycle as
\[
 H=x_0x_1\cdots x_{n-1}x_0,
 \qquad x_jx_{j+1}\text{ colored }c_j.
\]
For every $p$, reverse the block between $x_p$ and $x_{p+b+2}$:
\[
 x_p x_{p+1}\cdots x_{p+b+1}x_{p+b+2}
 \longmapsto
 x_p x_{p+b+1}x_{p+b}\cdots x_{p+1}x_{p+b+2}.
\]
Color the two new boundary edges by $c_p$ and $c_{p+b+1}$ and keep the
old colors on the reversed middle edges.  Call the resulting rainbow Hamilton
cycle $H'$.  The old edge $x_px_{p+1}$ is now a supported type-$b$
chord.  Hence $S_b(H')\ne\varnothing$ on the new cycle $H'$, so the equality
structure again gives $S_b(H')=\mathbb Z_n$.  Thus these switches may be
iterated.

We next show that every type-$(t-b)$ chord is absent from every member.
Suppose
$
 f=x_ix_{i+t-b+1}\in G_r.
$
The block replaced by $f$ has $t-b+1$ Hamilton edges, while its
complement has $b+3$ edges.  If the Hamilton edge of color $r$ lies in
the deleted block, choose any $(b+1)$-edge block in the complement.  If
it lies in the complement, choose such a block containing that edge.
Its type-$b$ shortcut is available in both boundary colors.  Those two
boundary colors are the colors of the first and last Hamilton edges of
the chosen block; since the block has at least two edges, at least one of
them is different from $r$, and that color is released by the block.
Together with $f$ colored $r$, the two edge-disjoint replacements delete
$(t-b)+b=t$ vertices and, by Lemma~\ref{lem:disjoint-shortcuts}, yield a
rainbow $C_4$, a contradiction.
Therefore all type-$(t-b)$ chords are absent in every color; the same is
true on every Hamilton cycle reached by the switches.

For $b=1$ the switch exchanges two consecutive positions, so the
switches generate all vertex permutations.  Any pair can then be placed
at the ends of a forbidden type-$(t-1)$ chord, forcing every member to
be edgeless.  For $b=2$ the switch exchanges positions at distance two
and generates all permutations within the two parity classes.  A
forbidden type-$(t-2)$ chord has odd endpoint distance $t-1=n-5$.
Hence every pair from opposite parity classes is absent from every
member, contradicting the Hamilton cycle.  This proves the claim.

Finally, apply \eqref{eq:C4-pair-equality} with $a=1$.  The claim rules
out $S_1\ne\varnothing$, so
$
 S_{t-1}=\mathbb Z_n.
$
Reverse
$
 x_0x_1\cdots x_tx_{t+1}$
 to
 $x_0x_tx_{t-1}\cdots x_1x_{t+1}$,
using colors $0$ and $t$ on the two new boundary edges and the old
colors on the reversed middle edges.  Call the resulting rainbow Hamilton
cycle $H'$.  On $H'$, the old type-$(t-1)$ chord $x_{-1}x_{t-1}$ joins vertices three
positions apart and is present in the boundary color $-1$.  Thus $S_2(H')\ne\varnothing$, contradicting the exchange principle.
\end{proof}

\subsection{Odd shortening: robust type-\texorpdfstring{$2$}{2} exchange}

For odd $t$ the count may miss equality, but by at most one unit when $n$ is
odd and two units when $n$ is even.  The next lemma shows that this still
leaves enough type-$2$ switches.  We keep only the condition
$S_2(H)\ne\varnothing$. The exceptional position may move when
the Hamilton cycle changes.

\begin{lemma}[Type-$2$ switches near equality]
\label{lem:near-full-step2}
Let $3\le t\le n-3$ be odd and suppose that there is no rainbow $C_{n-t}$.
Assume $n-t\ge4$ when $n$ is odd and $n-t\ge5$ when $n$ is even.  Then there
is a rainbow Hamilton cycle $H_0$ for which $S_2(H_0)\ne\varnothing$.
For every cycle $H$ reached from $H_0$ by the switches below and satisfying
$S_2(H)\ne\varnothing$, the following hold.
\begin{enumerate}
\item If $n$ is odd, every type-$2$ switch is available\footnote{Here ``available'' means that the displayed local exchange can be colored
	rainbow while all edges outside the exchanged segment retain their colors.}.
\item If $n$ is even, among switches whose starting positions have a fixed
parity, at most one is unavailable.
\item Every available switch can be colored so that the resulting Hamilton
cycle $H'$ again satisfies $S_2(H')\ne\varnothing$.
\end{enumerate}

\end{lemma}

\begin{proof}
\smallskip
\noindent\emph{Almost-full complementary supports.}
Write $t=2m+1$.  From \eqref{eq:odd-lower-count} and
\eqref{eq:boundary-gap},
\[
 0\le mn-\sum_{a=1}^{2m}|S_a|
 \le
 \begin{cases}
  0,& n\text{ odd},\\
  1,& n\text{ even}.
 \end{cases}
\]
Since
\[
 mn-\sum_{a=1}^{2m}|S_a|
 =\sum_{a=1}^{m}\bigl(n-|S_a|-|S_{t-a}|\bigr),
\]
and every summand is nonnegative, each complementary pair has total support
$n$ when $n$ is odd and at least $n-1$ when $n$ is even.  If both $S_a$ and $S_{t-a}$ were
nonempty, \eqref{eq:pair-short} would give
$
 n-|S_a|-|S_{t-a}|\ge n-t-2,
$
which is at least $2$ for odd $n$ and at least $3$ for even $n$ in the
stated range.  This is impossible.  The two supports cannot both be
empty either.  Hence exactly one of $S_a,S_{t-a}$ is nonempty, and the
nonempty one has size $n$ when $n$ is odd and at least $n-1$ when $n$
is even.

\smallskip
\noindent\emph{Creating type-$2$ support.}
If $S_2\ne\varnothing$, take the current Hamilton cycle for $H_0$.
Otherwise $S_{t-2}$ is nonempty and has size at least $n-1$.  Put $b=t-2$.  Consider the reversal
\[
 x_p x_{p+1}\cdots x_{p+t-1}x_{p+t}
 \longmapsto
 x_p x_{p+t-1}x_{p+t-2}\cdots x_{p+1}x_{p+t}.
\]
Its two new boundary edges are the type-$b$ shortcuts starting at $p$ and
$p+1$, use them in colors $p$ and $p+t-1$, respectively.  We also require
the type-$b$ shortcut starting at $p-1$ in color $p-1$.  After the reversal,
this third chord joins positions $p-1$ and $p+2$, and hence is a supported
type-$2$ shortcut on the new Hamilton cycle.

Choose $p$ so that the three required incidences
\[
 x_px_{p+t-1}\in G_p,\qquad
 x_{p+1}x_{p+t}\in G_{p+t-1},\qquad
 x_{p-1}x_{p+t-2}\in G_{p-1}
\]
are all present.  If the nonempty complementary support has size $n-1$,
then $D_{\rm supp}=1$, so $2D_{\rm supp}$ already uses the entire
even-order defect budget.  Hence every crossing row is tight and every supported chord
occurs in both boundary colors.  The missing start $s$ can interfere
only with $p\in\{s-1,s,s+1\}$, so at most three choices are excluded.

Otherwise the type-$b$ support is full.  The hypotheses imply $n\ge7$
($n-t\ge4$ with $t\ge3$ in odd order, and even more in even order).  The
defect budget leaves at most two missing boundary-color incidences in total.  A missing left incidence
at a start $s$ can exclude only the choices $p=s$ and $p=s+1$, while a
missing right incidence can exclude only $p=s-1$.  Thus at most four values
of $p$ are excluded.  Since $n\ge7$, an admissible $p$ exists.  Performing the reversal with the indicated two boundary colors gives a
rainbow Hamilton cycle; call it $H_0$.  The third incidence gives a supported
type-$2$ shortcut on $H_0$, so $S_2(H_0)\ne\varnothing$.

\smallskip
\noindent\emph{How many switches can fail.}
Fix a reached colored Hamilton cycle $H$ with $S_2(H)\ne\varnothing$.
For this local calculation, globally relabel the colors so that the Hamilton
edge in position $i$ has color $i$, and recompute all defect quantities for
this normalized cycle.  Abbreviate $S_a(H)$ by $S_a$.
The complementary-pair argument gives
$
 |S_2|\ge n-1$,
 $S_{t-2}=\varnothing$.
For the switch at $p$, write
$
 f_p=x_px_{p+3}, f_{p+1}=x_{p+1}x_{p+4}.
$
If $|S_2|=n-1$, then $D_{\rm supp}=1$.  This is possible only for even
$n$, and \eqref{eq:defect-identity} forces all other defects to vanish.
Thus every supported type-$2$ chord occurs in both boundary colors and every
crossing row is tight.  If $s$ is the unique unsupported start, only the
switches starting at $s$ and $s-1$ can fail and  these starts have opposite
parities.

Assume now that $S_2=\mathbb Z_n$.  If all crossing rows are tight, then
\eqref{eq:qzero-adjacent-colors} gives
\[
 e_j\in G_{j-1}\cap G_j\cap G_{j+1}\qquad\forall j.
\]
Choose any available boundary color from $\{p,p+2\}$ for $f_p$ and from
$\{p+1,p+3\}$ for $f_{p+1}$.  The two choices are distinct.  The two
unused colors from $\{p,p+1,p+2,p+3\}$ can then be assigned to
$e_{p+2}$ and $e_{p+1}$, whose available color sets contain
$\{p+1,p+2,p+3\}$ and $\{p,p+1,p+2\}$, respectively: if $p$ is
unused, give it to $e_{p+1}$; if $p+3$ is unused, give it to $e_{p+2}$;
otherwise the two unused colors are $p+1,p+2$.  Hence every switch is
available.

It remains only to consider a positive crossing defect.  If $n$ is odd,
the total defect budget is $1$. Hence $D_{\rm col}=0$,  both preferred
shortcut incidences are present and the usual coloring
$(p,p+2,p+1,p+3)$ works.  Let $n$ be even.  Again there is nothing to
prove when $D_{\rm col}=0$.  Otherwise
\eqref{eq:defect-identity} forces
\[
 D_{\rm col}=D_{\rm cross}=1,
 \qquad D_{\rm deg}=D_{\rm supp}=0.
\]
Thus there is exactly one missing shortcut incidence and one unit of
crossing deficit.  Only the switch for which that incidence is a preferred
boundary color can be affected.  Suppose, for instance, that
$f_p\notin G_p$.  Then $f_p\in G_{p+2}$, while $f_{p+1}$ has both
boundary colors.  If $e_{p+1}\in G_p$, the switch is still colorable.  Use $p+2$ on
$f_p$ and $p$ on $e_{p+1}$.  If $e_{p+2}\in G_{p+1}$, use colors
$p+1,p+3$ on $e_{p+2},f_{p+1}$, respectively.  Otherwise $e_{p+2}\notin G_{p+1}$, so the contrapositive of the first
implication in \eqref{eq:qzero-adjacent-colors} gives $q_{p+2}>0$.
This is the unique crossing defect.  Since $t+2<n$, the row
$p+t+4$ is different from $p+2$ modulo $n$, and hence
$q_{p+t+4}=0$.  The second implication in
\eqref{eq:qzero-adjacent-colors}, with $j=p+2$, now gives
$e_{p+2}\in G_{p+3}$. Then use $p+3,p+1$ instead.  Hence this switch can fail only if
$e_{p+1}\notin G_p$.
The case of a missing preferred right boundary is symmetric.  Therefore at
most one switch in total is unavailable; in particular, each parity class
of starting positions contains at most one unavailable switch.

\smallskip
\noindent\emph{Persistence under switching.}
Every switch whose availability was established above admits a coloring
that also preserves the invariant $S_2\ne\varnothing$.  If
$f_p\in G_p$, use color $p$ on $f_p$ and complete the switch by the
coloring already given.  The only remaining guaranteed case is the one in
which $f_p\notin G_p$ but $e_{p+1}\in G_p$. There we use $p+2$ on
$f_p$ and $p$ on $e_{p+1}$, and the preceding argument supplies the two
remaining colors.  (If the missing preferred incidence is on the right,
$f_p$ still has color $p$.)

Let $H'$ be the resulting Hamilton cycle.  In every case color $p$ appears
on the first or third edge of the new four-edge segment.  The old edge
$x_px_{p+1}$ now joins positions $p$ and $p+3$ and still belongs to $G_p$,
hence it is a supported type-$2$ chord of $H'$.  Thus
$S_2(H')\ne\varnothing$,  the argument may be iterated. The exceptional
position (if any) may move.
\end{proof}

The relation $S_{t-2}=\varnothing$ excludes only the two boundary
colors of a complementary chord.  For the switching argument we need the
stronger conclusion that such a chord is absent from every member of the
system.

\Needspace{7\baselineskip}
\begin{lemma}[Complementary-chord exclusion]
\label{lem:boundary-arbitrary-forbidden}
Under the hypotheses of Lemma~\ref{lem:near-full-step2}, let $H$ be any
reached colored Hamilton cycle with $S_2(H)\ne\varnothing$.  After globally
relabeling the colors so that
$H=x_0x_1\cdots x_{n-1}x_0$ with $x_ix_{i+1}\in E(G_i)$, one has
\begin{equation}
 x_ix_{i+t-1}\notin E(G_r)
 \qquad\forall i,r\in\mathbb Z_n.
 \label{eq:boundary-arbitrary-forbidden}
\end{equation}
\end{lemma}

\begin{proof}
For this normalized cycle write $S_a=S_a(H)$ and
$e_j=x_jx_{j+1}$.  We know that $S_{t-2}=\varnothing$.
Suppose, to the contrary, that
$
 f=x_ix_{i+t-1}\in G_r.
$
Let $B=B_H(i,t-2)$ be the $(t-1)$-edge block replaced by $f$, and let
$R$ be the complementary Hamilton arc.  Thus $R$ has $n-t+1$ edges,
at least $5$ when $n$ is odd and at least $6$ when $n$ is even.  We find
a three-edge block $J\subseteq R$ whose type-$2$ shortcut can be used
together with $f$.

\smallskip
\noindent\emph{If $e_r\in B$.}
Color $f$ with the released color $r$.  At most one type-$2$ start is
unsupported, so $R$ contains a supported three-edge block $J$.  Any legal
boundary color of its shortcut is different from $r$, because both boundary
edges of $J$ lie in $R$.  The two disjoint replacements therefore give a
rainbow $C_{n-t}$ by Lemma~\ref{lem:disjoint-shortcuts}.

\smallskip
\noindent\emph{If $e_r$ is an interior edge of $R$.}
Choose a three-edge block $J\subseteq R$ containing $e_r$.  If all type-$2$
starts are supported, take the block centered at $e_r$; its boundary colors
are both different from $r$.  If one start is unsupported, then every other
type-$2$ shortcut occurs in both boundary colors.  Among two consecutive
three-edge blocks containing $e_r$, at least one is supported, and one of
its boundary colors is different from $r$.  Use that color on the shortcut
of $J$ and color $f$ with $r$.  Again Lemma~\ref{lem:disjoint-shortcuts}
gives a rainbow $C_{n-t}$.

\smallskip
\noindent\emph{If $e_r$ is an end edge of $R$.}
Orient $R$ so that $e_r$ is its first edge, and let $J_0,J_1$ be the first
two three-edge blocks of $R$.  If the shortcut of $J_0$ has a legal boundary
color different from $r$, use it and color $f$ with $r$, as above.

It remains only when $J_0$ is unsupported, or when its non-$r$ boundary
incidence is missing.  If $J_0$ is unsupported, this is impossible for odd $n$ and exhausts
the two-unit defect budget for even $n$. Hence every crossing row is tight
and every other supported type-$2$ shortcut has both boundary colors.  Since
$e_r$ is the first edge of $R$, the preceding edge $e_{r-1}$ is the last
edge of $B$.  Tightness and \eqref{eq:qzero-adjacent-colors} give
$e_r\in G_{r-1}$.  Recolor $e_r$ with the released color $r-1$, color $f$
with $r$, and use either boundary color of the $J_1$ shortcut.

Now suppose that all type-$2$ starts are supported but the non-$r$ boundary
incidence of $J_0$ is missing.  If $e_r\in G_{r-1}$, the same recoloring
works.  Otherwise the contrapositive of the first implication in
\eqref{eq:qzero-adjacent-colors} gives $q_r>0$.  Together with the missing
boundary incidence, \eqref{eq:defect-identity} rules out odd $n$. For even
$n$ it forces these to be the only two defects.  In particular every
other type-$2$ shortcut, including the shortcut of $J_1$, has both boundary
incidences.  Since $t+2<n$, the row $r+t+2$ is different from the defective
row $r$, we have $q_{r+t+2}=0$. The second implication in
\eqref{eq:qzero-adjacent-colors}, with $j=r$, gives $e_r\in G_{r+1}$.
Recolor $e_r$ with the color $r+1$ released by $J_1$, color $f$ with $r$,
and color the $J_1$ shortcut with its other boundary color.  The opposite end of $R$ is
symmetric.

In the endpoint recoloring subcases, the recolored edge $e_r$ is retained,
and its new color is released by $B$ or $J_1$. The shortcut of $J_1$ receives
a different released boundary color, while $f$ receives the now free color
$r$.  Thus the resulting cycle is rainbow.  In every case the two disjoint
replacements delete $(t-2)+2=t$ internal vertices, so the resulting cycle
has length $n-t$, a contradiction.
\end{proof}

\begin{lemma}[Odd $t$ at the lower threshold]
\label{lem:boundary-odd-t}
Let $3\le t\le n-3$ be odd.  Then the system contains a rainbow
$C_{n-t}$.
\end{lemma}

\begin{proof}
If $n$ is even and  $n-t=3$, then  the system contains a rainbow
$C_{3}$ use
Lemma~\ref{lem:boundary-triangle}.  Otherwise the hypotheses of
Lemma~\ref{lem:near-full-step2} hold: when $n$ is odd, the oddness of
$t$ makes $n-t$ even, hence $n-t\ge4$; when $n$ is even, $n-t$ is odd,
and after excluding $n-t=3$ we have $n-t\ge5$.  Fix the cycle $H_0$ supplied
there and keep only the switches allowed by that lemma.

If $n$ is odd, every switch is available.  The reached Hamilton cycles
are therefore closed under \eqref{eq:abstract-switch}, while
\eqref{eq:boundary-arbitrary-forbidden} gives the forbidden-pair condition
of Lemma~\ref{lem:exchange-closure} with $q=t-1$.  That lemma forces $n$
to be even, a contradiction.

Let $n=2h$ be even.  Here $n-t\ge5$ and $t\ge3$, so $n\ge8$ and
$h\ge4$.  Moreover
$
  1\le d:=\frac{t-1}{2}\le h-2,
$
because $t\le n-3=2h-3$.  The switches preserve the two vertex parity
classes $X,Y$.  For any ordered pair of distinct vertices $u,v$ in one
parity class, measure their oriented separation in the cyclic order obtained
by adding $2$.  If this separation is smaller than $d$, it can be increased
by one either by moving $u$ one step backwards or by moving $v$ one step
forwards. If it is larger than $d$, then there are two analogous moves that
decrease it.  The two candidate switches have the same starting-position
parity and are distinct.  By Lemma~\ref{lem:near-full-step2}, at least
one is available.  Repeating, we can place $u,v$ in positions differing
by $t-1$.  Equation \eqref{eq:boundary-arbitrary-forbidden} then shows
that every pair inside $X$, and likewise every pair inside $Y$, is
nonadjacent in every member.  Hence
$
 G_r\subseteq K_{X,Y}$ for all $r$.
Both parts have size $h$.  If any member misses a cross edge $xy$, then
$\delta(x),\delta(y)\le h-1$ and all other system degrees are at most
$h$, giving
\[
 \sum_v\delta(v)\le2h^2-2=\frac{n^2}{2}-2,
\]
contrary to \eqref{eq:boundary-layer}.  Thus every member is
$K_{X,Y}$, which instead gives system-degree sum $n^2/2$, again a
contradiction.
\end{proof}

\subsection{The boundary length \texorpdfstring{$C_{n-1}$}{C(n-1)} at the lower threshold}

For $C_{n-1}$ it is simpler to work directly with the crossing rows
obtained by deleting one vertex from the Hamilton cycle.

\begin{lemma}[Clean-parity exchange]\label{lem:clean-parity-exchange}
Let $n=2h\ge6$, and let
$
 H=x_0x_1\cdots x_{n-1}x_0
$
be a colored rainbow Hamilton cycle in a system with no rainbow
$C_{n-1}$.  Suppose that for some $\sigma\in\{0,1\}$,
\begin{equation}
 \delta(x_{k-1})+\delta(x_{k+1})=n
 \qquad(k\equiv\sigma\pmod2).
 \label{eq:clean-parity-degree}
\end{equation}
Then every step-$2$ switch that exchanges the vertices in positions
$p+1$ and $p+3$ with $p+1\equiv\sigma\pmod2$ is available, and these
switches may be iterated.  Consequently, the set
$\{x_i:i\equiv\sigma\pmod2\}
$
 is independent in every member of the system.
\end{lemma}

\begin{proof}
Call the indices $k\equiv\sigma\pmod2$ \emph{clean}.  For such $k$,
the calculation leading to \eqref{eq:t1-cross}. Together with
\eqref{eq:clean-parity-degree}, forces equality in the deleted-vertex
crossing row.  If $e_j=x_jx_{j+1}$ has color $c_j$, the first end position
of the clean row at $k$ gives $e_{k+1}\in G_{c_k}$, while the last end
position of the clean row at $k+2$ gives $e_k\in G_{c_{k+1}}$.  Together
with the original Hamilton colors, this yields
\begin{equation}
 e_k,e_{k+1}\in G_{c_k}\cap G_{c_{k+1}}.
 \label{eq:clean-two-color-block}
\end{equation}
Thus every clean pair $e_k,e_{k+1}$ is a two-color block, its two colors
may be interchanged.

Let $p+1$ be clean.  In the clean row obtained by deleting $x_{p+1}$,
tightness at the second crossing position gives one of two alternatives.
After deleting $x_{p+3}$, the alternative
$
 x_{p+2}x_{p+4}\in G_{c_{p+1}}
$
  allow us to recolor $e_{p+1}$ with the
released color $c_{p+2}$ by \eqref{eq:clean-two-color-block} and color the
new chord with $c_{p+1}$.  This gives a rainbow $C_{n-1}$, a contradiction.
Hence
$
 x_px_{p+3}\in G_{c_p}.
$
The symmetric clean row at $p+3$ gives
$
 x_{p+1}x_{p+4}\in G_{c_{p+3}}.
$
Therefore the standard switch
\[
 x_p x_{p+1}x_{p+2}x_{p+3}x_{p+4}
 \longmapsto
 x_p x_{p+3}x_{p+2}x_{p+1}x_{p+4}
\]
is rainbow, with colors
$ c_p,c_{p+2},c_{p+1},c_{p+3}$.  The switch exchanges two vertices of the
same parity, so every vertex in the opposite parity positions is fixed.
Moreover, the set of two colors carried by each clean two-edge block is
unchanged: the middle affected block merely exchanges $c_{p+1}$ and
$c_{p+2}$, while the two adjacent clean blocks retain their original color
sets because the new boundary edges receive $c_p$ and $c_{p+3}$.  The
vertices in the opposite parity positions are fixed, so
\eqref{eq:clean-parity-degree} also persists.  Hence the clean-block
structure may be re-established after every switch, and the same argument
iterates.

These switches are adjacent transpositions in the cyclic order of the
positions of parity $\sigma$; hence they generate arbitrary permutations
of the vertices of $X$.  The clean two-edge blocks are pairwise disjoint
and their color sets partition all $n$ colors.  Given distinct $u,v\in X$
and a color $r$, choose the clean block $e_k,e_{k+1}$ whose color set
contains $r$, and move $u,v$ to positions $k$ and $k+2$.  If $uv\in G_r$,
deleting $x_{k+1}$ deletes both edges of the block, releases color $r$,
and gives a rainbow $C_{n-1}$.
Hence $uv\notin G_r$.  Since $u,v$ and $r$ were arbitrary, $X$ is
independent in every member.
\end{proof}

\Needspace{7\baselineskip}
\begin{lemma}[The lower-threshold $C_{n-1}$]
\label{lem:boundary-nminus1}
Under the standing assumptions of this section, the system contains a rainbow
$C_{n-1}$.
\end{lemma}

\begin{proof}
For $n=4$ the assertion is Lemma~\ref{lem:boundary-triangle}.  Assume
$n\ge5$ and fix a colored rainbow Hamilton cycle
$H=x_0x_1\cdots x_{n-1}x_0$.  For each $k$, delete $x_k$ and write
$
 P_k:=y_1y_2\cdots y_{n-1}$ with $y_j=x_{k+j}$.
The unused colors on $P_k$ are $k$ and $k-1$.  Define
\[
 A_k:=\{j\in[n-2]:y_1y_{j+1}\in E(G_k)\},
 \qquad
 B_k:=\{j\in[n-2]:y_jy_{n-1}\in E(G_{k-1})\},
\]
and put
\[
 q_k:=(n-2)-|A_k|-|B_k|\ge0,
 \qquad
 \varepsilon_k:=n-\delta(x_{k+1})-\delta(x_{k-1}).
\]
The crossing inequality gives
\begin{equation}
 0\le q_k\le\varepsilon_k,
 \qquad
 \sum_k\varepsilon_k=n^2-2B_n.
 \label{eq:nminus1-count}
\end{equation}
Moreover, if $\varepsilon_k=0$, then $q_k=0$ and both layer degrees used
in the $k$th crossing row are equal to the corresponding system degrees.

\Needspace{5\baselineskip}
\smallskip
\noindent\emph{Odd order.}
Assume first that $n$ is odd.  Then \eqref{eq:boundary-gap} gives
$\sum_k\varepsilon_k=1$.  Hence there is at most one
index $k_0$ for which $q_{k_0}>0$. If it occurs, then $q_{k_0}=1$.
For every other $k$, the crossing sets partition $[n-2]$.  More precisely, for a tight row the first end cell gives the relevant
inclusion $e_j\in G_{j-1}$ after reindexing, while the last end cell gives
the inclusion $e_j\in G_{j+1}$.  Across all rows, the first-end cells are
exactly the cells used for the first family of inclusions and the last-end
cells are exactly those used for the second family.  There is at most one
missing cell in the entire crossing table.  If there is none, both families
hold.  If the missing cell is an interior cell, neither family is affected.
If it is a first-end or last-end cell, only the corresponding family can
fail.  Hence at least one of
\[
 e_j\in G_{j-1}\quad\forall j,
 \qquad\text{or}\qquad
 e_j\in G_{j+1}\quad\forall j
\]
holds globally.

Now suppose that a distance-two chord $f=x_ix_{i+2}$ belongs to an
arbitrary $G_r$.  Delete $e_i,e_{i+1}$.  If
$r\in\{i,i+1\}$, color $f$ with $r$ and keep all retained Hamilton edges in
their original colors, immediately obtaining a rainbow $C_{n-1}$.
Assume therefore that $r\notin\{i,i+1\}$.

If the globally clean propagation is
$
 e_j\in G_{j-1}$ for all $j$,
start with $e_{i+2}$ and move forward along the retained Hamilton arc toward
$e_r$, recoloring each encountered edge $e_j$ with color $j-1$.  The first
new color, $i+1$, was released by deleting $e_{i+1}$. Thereafter the new
color on $e_j$ is precisely the old color just released from the preceding
edge.  When $e_r$ is recolored with $r-1$, color $r$ becomes free.

If instead the globally clean propagation is
$
 e_j\in G_{j+1}$ for all $j$,
perform the symmetric shift in the reverse direction: start with $e_{i-1}$,
recolor it with the released color $i$, and continue backwards, recoloring
$e_j$ with $j+1$ until $e_r$ is reached.  Again the colors remain distinct
throughout and the old color $r$ is released at the last step.  In either
case assigning color $r$ to $f$ gives a rainbow $C_{n-1}$, a contradiction.
Hence all distance-two chords are absent from every member.

For $n=5$, choose any crossing row with $q_k=0$.  At its middle position
both candidates are distance-two chords, contradicting the preceding
prohibition.  Let $n\ge7$.  For the standard step-$2$ switch at position
$p$, the chord $x_px_{p+3}$ in color $p$ is forced by the cell $j=2$ of
the crossing row obtained by deleting $x_{p+1}$: the competing candidate in
that cell is a forbidden distance-two chord.  Symmetrically,
$x_{p+1}x_{p+4}$ in color $p+3$ is forced by the cell $j=n-3$ of the row
obtained by deleting $x_{p+3}$.  Thus a switch at $p$ fails only if the possible missing crossing cell is
one of the two cells
\[
 (\text{row }p+1,\text{ position }2),
 \qquad
 (\text{row }p+3,\text{ position }n-3).
\]
A single missing cell can block at most one switch: to block the first
role its within-row position must be $2$, whereas to block the second it
must be $n-3$. These are distinct because $n\ge7$.  Hence the standard
step-$2$ Hamilton switch is available at every position except possibly one.

This conclusion can be re-established after every performed switch.
Indeed, \eqref{eq:nminus1-count} is valid for every colored rainbow
Hamilton cycle, the sum $\sum_k\varepsilon_k$ remains $1$, and the preceding
propagation argument again forbids every distance-two chord in every
layer.  Thus on each reached cycle there is again at most one blocked
step-$2$ switch.  Since $n$ is odd, addition by $2$ is a single cyclic order on the
positions.  For two prescribed vertices, take the shorter separation in
this order.  If it is larger than one, there are two distinct boundary
transpositions that decrease it by one: move the first vertex one step
towards the second, or move the second one step towards the first.  At
most one switch is blocked, so one of these two moves is available.
Iterating makes the two vertices adjacent in the step-$2$ order, i.e. they
occupy positions at distance two on the Hamilton cycle.  Such a pair is forbidden in every member, so all
members would be edgeless, a contradiction.

\Needspace{5\baselineskip}
\smallskip
\noindent\emph{Even order.}
Now let $n=2h$, so $\sum_k\varepsilon_k=2$.  For
$\sigma\in\{0,1\}$,
\[
 \sum_{k\equiv\sigma\ (2)}\varepsilon_k
 =hn-2\sum_{i\equiv1-\sigma\ (2)}\delta(x_i).
\]
This is a nonnegative even integer.  Since the two parity sums add to $2$,
one of them is zero.  Fix $\sigma$ with
$
 \varepsilon_k=0 \ (k\equiv\sigma\pmod2),
$
and put
\[
 X:=\{x_i:i\equiv\sigma\pmod2\},\qquad Y=V\setminus X.
\]
Lemma~\ref{lem:clean-parity-exchange} shows that $X$ is independent in
every member.
The vanishing parity sum gives
\[
 \sum_{y\in Y}\delta(y)=h^2,
 \qquad
 \sum_{x\in X}\delta(x)=h^2-1.
\]
Since $X$ is independent, $\delta(x)\le h$ for $x\in X$.  Hence one
vertex $x^*\in X$ has system degree $h-1$, while every vertex of
$X\setminus\{x^*\}$ has system degree $h$.  It follows that every edge
between $X\setminus\{x^*\}$ and $Y$ belongs to every member of the
system.

If some $G_r$ contained an edge $yy'$ inside $Y$, use $yy'$ in color $r$
and alternate through all vertices of $Y$ and all vertices of
$X\setminus\{x^*\}$.  The remaining $n-2$ edges are common cross edges,
so they can receive the other $n-2$ colors.  This gives a rainbow
$C_{n-1}$, a contradiction.  Thus $Y$ is also independent in every
member, and therefore
$
 G_r\subseteq K_{X,Y}$ for all $r$.
If some cross edge is missing, then the system-degree sum is at most
$n^2/2-2$. If none is missing, it is $n^2/2$.  Both contradict
\eqref{eq:boundary-layer}.
\end{proof}

\subsection{Completion at the lower threshold}

\begin{proposition}[Pancyclicity at the lower threshold]
\label{prop:one-unit-stability}
If \eqref{eq:boundary-layer} holds and the system contains a rainbow
Hamilton cycle, then the system is rainbow pancyclic.
\end{proposition}

\begin{proof}
If $n=3$, the standing rainbow Hamilton cycle is already the only cycle
length required for rainbow pancyclicity.  Assume $n\ge4$.
Lemma~\ref{lem:boundary-triangle} gives $C_3$ and the standing Hamilton
cycle gives $C_n$.  Lemma~\ref{lem:boundary-nminus1} gives $C_{n-1}$.
For $3\le\ell\le n-2$, let $t=n-\ell$,  then
$2\le t\le n-3$. If $t$ is odd, automatically $t\ge3$.  Thus the
parameter hypotheses of Lemmas~\ref{lem:boundary-even-t} and
\ref{lem:boundary-odd-t} are satisfied, according to the parity of $t$.
Hence every length from $3$ to $n$ occurs.
\end{proof}

\begin{proof}[Proof of Theorem~\ref{thm:bondy}]
Set $S=\sum_v\delta(v)$.  Since $S$ is an integer and
$S\ge B_n=\lceil n^2/2\rceil-1$, either $S=B_n$ or the next possible
integer value already satisfies $S\ge n^2/2$ (strictly so when $n$ is
odd).
If $S=B_n$, Proposition~\ref{prop:one-unit-stability} applies and the
system is rainbow pancyclic.  If $S\ge n^2/2$, apply
Proposition~\ref{prop:pancy-baseline}, its only non-pancyclic outcome is
the common balanced complete bipartite system.  These two cases exhaust
all possible integer values of $S$ under the hypothesis.
\end{proof}

\section{Extremal examples and sharpness}
\label{sec:sharpness}

There are two extremal points to record.  At the stated density the only
non-pancyclic system is the common balanced complete bipartite graph.  When
$n$ is even, deleting one cross edge from a single member gives a
counterexample at the next lower integer level.

\begin{proposition}[The exceptional family]
\label{prop:extremal-example}
Let $n=2m\ge4$ and let $V=X\sqcup Y$ with $|X|=|Y|=m$.  If
\[
 G_1=\cdots=G_n=K_{X,Y},
\]
then the system contains a rainbow Hamilton cycle and a rainbow
$C_\ell$ for every even $\ell$ with $4\le \ell\le n$, but it contains no
odd cycle.  In particular it is exactly the non-pancyclic alternative
allowed by Theorem~\ref{thm:bondy}.
\end{proposition}

\begin{proof}
Every edge of $K_{X,Y}$ belongs to every member.  Hence any alternating
Hamilton cycle can be assigned the $n$ distinct labels and is rainbow.
Likewise, choosing $k$ vertices from each part gives a rainbow $C_{2k}$
for every $2\le k\le m$.  The common bipartition forbids every odd cycle.
\end{proof}

\begin{proposition}[Sharpness witness at even order]
\label{prop:threshold-minus-two}
Let $n=2m\ge4$, let $V=X\sqcup Y$ with $|X|=|Y|=m$, and fix an edge
$e\in E(K_{X,Y})$.  If
\[
 G_1=K_{X,Y}-e,
 \qquad
 G_2=\cdots=G_n=K_{X,Y}.
\]
Then the system contains a rainbow Hamilton cycle and
$
 \sum_{v\in V}\delta(v)=\frac{n^2}{2}-2,
$
but it is not rainbow pancyclic and is not the exceptional system in
Theorem~\ref{thm:bondy}.  Consequently the threshold in
Theorem~\ref{thm:bondy} cannot be replaced by $n^2/2-2$ while retaining
the same conclusion and the same exceptional family.
\end{proposition}

\begin{proof}
Let $K=K_{X,Y}$.  Choose any Hamilton cycle $C$ of $K$.  If $e\notin
E(C)$, every edge of $C$ belongs to every member, so assign the $n$
labels arbitrarily.  If $e\in E(C)$, choose an edge $f\in E(C)\setminus
\{e\}$.  Give $f$ label $1$, give $e$ any label from $\{2,\ldots,n\}$,
and assign the remaining labels bijectively to the remaining edges of
$C$.  All these assignments are legal, so $C$ is rainbow.

The two ends of $e$ have system degree $m-1$ and all other vertices have
system degree $m$.  Hence
\[
 \sum_v\delta(v)=2(m-1)+(2m-2)m=2m^2-2=\frac{n^2}{2}-2.
\]
Every member is a subgraph of the same bipartite graph $K$, so the
system has no odd cycle.  Since $G_1\ne K$, it is not the exceptional
system of Theorem~\ref{thm:bondy}.
\end{proof}

For even $n$, Proposition~\ref{prop:threshold-minus-two} proves exact
integer sharpness: the theorem holds at $n^2/2-1$ and fails at
$n^2/2-2$ with the same exceptional family.

\paragraph{The odd-order threshold.}
The corresponding sharp threshold for odd order remains open.  More
precisely, for odd $n$ let $b_n$ be the least integer $B$ such that every
$n$-member graph system on a common $n$-vertex set which contains a rainbow
Hamilton cycle and satisfies
\[
   \sum_{v\in V}\min_{i\in[n]} d_{G_i}(v)\ge B
\]
is rainbow pancyclic.  Theorem~\ref{thm:bondy} gives
$
   b_n\le \frac{n^2-1}{2}.
$
Determining $b_n$ is a natural sharpness problem.  The ordinary diagonal
case does not settle it, since the aggregate system degree allows the
minimizing member to vary from vertex to vertex.

\section*{Declaration of generative AI and AI-assisted technologies in the manuscript preparation process}
During the preparation and internal verification of this work, the authors used
OpenAI ChatGPT to assist with manuscript organization and exposition,
literature-search support, and internal consistency checks during proof
development.  The authors independently reviewed and verified all mathematical
claims, proofs, references, and final wording, edited the content as needed,
and take full responsibility for the content of the publication.

\end{document}